\documentclass[final,onefignum,onetabnum]{siamart171218}
\usepackage{enumerate}
\usepackage{amsmath, amssymb}
\usepackage{float}
\usepackage{subcaption}
\usepackage{flexisym}
\usepackage{mathtools}
\usepackage{breqn}
\usepackage{array}
\usepackage{tikz}
\usepackage{cite}
\usetikzlibrary{shapes.geometric, arrows}
\usepackage[font = small]{caption}
\usepackage{paralist}
\usepackage{mathrsfs}
\usepackage{algorithm}
\usepackage{algpseudocode}
\usepackage[inline]{enumitem}
\usepackage{bbm}
\allowdisplaybreaks[1]

\theoremstyle{definition}
\newtheorem{assumption}{Assumption}

\DeclareMathOperator*{\argmax}{arg\,max}

\DeclareMathOperator*{\diag}{diag}
\DeclareMathOperator*{\blkdiag}{blkdiag}

\title{Resilient Distributed Field Estimation \thanks{\funding{This material is based upon work supported by DARPA under agreement number FA8750-16-2-0033, by the Department of Energy under award number DE-OE0000779, and by the National Science Foundation under Award Number CCF 1513936.}}}
\author{Yuan Chen \thanks{Department of Electrical and Computer Engineering, Carnegie Mellon University, Pittsburgh, PA (\email{yuanche1@andrew.cmu.edu}, \email{soummyak@andrew.cmu.edu}, \email{moura@andrew.cmu.edu}).}
\and{Soummya Kar \footnotemark[2]} 
\and {Jos\'{e} M. F. Moura}\footnotemark[2]}

\begin{document}

\maketitle

\begin{abstract}
	We study resilient distributed field estimation under measurement attacks. A network of agents or devices measures a large, spatially distributed physical field parameter. An adversary arbitrarily manipulates the measurements of some of the agents. Each agent's goal is to process its measurements and information received from its neighbors to estimate only a few specific components of the field. We present \textbf{SAFE}, the Saturating Adaptive Field Estimator, a consensus+innovations distributed field estimator that is resilient to measurement attacks. Under sufficient conditions on the compromised measurement streams, the physical coupling between the field and the agents' measurements, and the connectivity of the cyber communication network, \textbf{SAFE} guarantees that each agent's estimate converges almost surely to the true value of the components of the parameter in which the agent is interested. Finally, we illustrate the performance of \textbf{SAFE} through numerical examples. 
\end{abstract}


\section{Introduction}\label{sect: intro}
This paper is concerned with resilient distributed field estimation under measurement attacks. We design a distributed estimator that we refer to as the Saturating Adaptive Field Estimator (\textbf{SAFE}) of a large dimensional unknown field when an adversary arbitrarily changes some of the agents' measurements. In contrast to existing work on resilient distributed estimation, each agent is only interested in estimating a fraction of the field, instead of the entire field. Under appropriate conditions, we prove strong consistency of each agent's estimate. When the unknown field is very large, \textbf{SAFE} dramatically reduces the local processing and communication at each agent, making it better suited in practical scenarios.

\subsection{Motivation}
The Internet of Things (IoT) brings about many multi-agent systems applications, such as road side unit (RSU) networks for monitoring traffic and controlling traffic signals in smart cities~\cite{TrafficRSU}, teams of robots mapping an unknown environment~\cite{RobotMapping}, or state estimation in the power grid~\cite{PowerGrid}. In these applications, the individual devices measure a component or a few components of a spatially distributed field and process their data to learn about their physical surroundings. 
For example, in traffic management, the network of road side units estimates traffic conditions throughout the entire city (a large, spatially distributed field) to effectively control traffic signals~\cite{TrafficRSU, ConnectedRoads}. 

This paper studies distributed field estimation~\cite{SahuRandomFields}. A team of agents or devices makes repeated (over time) local measurements of an unknown spatially distributed field. 
That is, each agent's measurements are \textit{physically coupled} to only a few components of the field. The agents process their measurements and the information they receive from neighbors over a \textit{cyber} communication network. Each agent only estimates a portion of the unknown field, because, due to the size of the field, it is infeasible for each individual agent to estimate the \textit{entire} field. This contrasts with common distributed estimation settings where, even though the agents make local measurements, the goal of each agent is to estimate the entire field. {\color{black} Certain IoT setups involve large numbers of devices, distributed throughout a large physical environment, so it is unrealistic for every agent to estimate the entire field.} In multi-robot navigation and mapping, for example, each robot only measures its local surroundings (instead of the entire unknown environment), and, through collaboration with close-by neighbors, also only estimates the field in nearby locations.  {\color{black} Similarly, in smart cities, each road side unit (RSU) may be tasked with monitoring the traffic within a local neighborhood, instead of estimating the state of traffic throughout the entire city.}

Security is a prominent challenge in deploying IoT applications as individual devices are vulnerable to adversarial cyber-attacks. Autonomous vehicles and mobile robots, for example, may fall victim to cyber-attacks that manipulate its onboard sensors and hijack its control systems~\cite{VehiclesExample2, HACMS}. Without adequate damage mitigation, compromised devices jeopardize the functionality and reliability of these team based systems. In this paper, we focus on data and measurement attacks, where an adversary manipulates a subset of the agents' measurements, arbitrarily changing their data. Our goal is to ensure that \textit{all} of the agents, even those with compromised measurement streams, consistently estimate the fraction of the field in which they are interested. To this end, we present \textbf{SAFE}, the Saturating Adaptive Field Estimator, a distributed field estimator that is resilient to measurement attacks. 

\subsection{Literature Review}

Existing work in resilient multi-agent computation has focused on problems where the agents share a common (homogeneous) processing objective. In the classic Byzantine generals problem, a group of loyal agents must agree on whether or not to attack an enemy city while treacherous agents attempt to disrupt the decision-making process~\cite{ByzantineGenerals}. The agents pass messages to each other, following a specific protocol in an all-to-all communication scheme, and reach a consensus on their decision as long as more than two thirds of the agents are loyal~\cite{ByzantineGenerals, SIAMByzantine}. References~\cite{ResilientConsensus, LeBlanc1, Pasqualetti1} study resilient multi-agent consensus in the presence of adversaries over sparse (i.e., not all-to-all) communication topologies. One application of resilient consensus is robot gathering~\cite{SIAMConsensus}, where a team of robots must rendezvous at a particular location. Reference~\cite{SIAMConsensus} designs a robot gathering algorithm that may tolerate Byzantine failures. In these consensus problems~\cite{ByzantineGenerals, ResilientConsensus, LeBlanc1, Pasqualetti1, SIAMConsensus}, the (loyal) agents share the same goal: to reach agreement on a decision or value.

Beyond consensus, prior work in resilient inference has also focused on problems where the agents have a common objective. In inference tasks, agents process their measurements to perform hypothesis testing or to recover an unknown parameter. This differs from consensus tasks, where there are no measurements involved and agents only need to reach agreement. References~\cite{Fawzi, Shoukry, Pasqualetti2} address resilient state estimation under data attacks in centralized, single-agent systems. 
References~\cite{Varshney1, Varshney2, FusionCenterEstimation1, Zhang2, SIAMChangeDetection} study decentralized Byzantine inference, where agents transmit local data or decisions to a fusion center for processing. Adversarial (Byzantine) agents transmit falsified data to the fusion center whose goal is to recover the parameter without being misled by the adversarial agents. 

Our prior work has addressed resilient estimation in fully distributed settings (no fusion center) in the presence of Byzantine agents~\cite{ChenIoT, ChenDistributed1} and measurement attacks~\cite{ChenSAGE}. In~\cite{ChenIoT, ChenDistributed1, ChenSAGE}, each agent in a network of agents makes partial measurements of an unknown parameter and all of the agents attempt to estimate the \textit{entire} parameter in the presence of adversarial attacks. That is, all of the agents share the same estimation goal. In distributed settings, aside from resilient estimation, another relevant problem is resilient optimization~\cite{Sundaram3, SIAMFormationControl}, where a team of agents, each with access to its own \textit{local} objective function, cooperate to optimize a single \textit{global} objective (e.g., the sum of their local objective functions). Reference~\cite{Sundaram3} designs countermeasures to mitigate the damage from adversarial agents in distributed optimization tasks, and reference~\cite{SIAMFormationControl} solves attack resilient vehicular formation control as a distributed optimization problem.

\subsection{Summary of Contributions}
Unlike the existing work in resilient consensus~\cite{ByzantineGenerals, ResilientConsensus, LeBlanc1, Pasqualetti1, SIAMConsensus, SIAMByzantine}, centralized and decentralized inference~\cite{Fawzi, Shoukry, Pasqualetti2, Varshney1, Varshney2, FusionCenterEstimation1, Zhang2, SIAMChangeDetection}, distributed parameter estimation~\cite{ChenIoT, ChenDistributed1, ChenSAGE} and optimization~\cite{Sundaram3, SIAMFormationControl}, which all study setups with a common processing objective, in this paper, we consider the case when agents have different heterogeneous estimation goals. In particular, we study resilient distributed field estimation, where each agent seeks to estimate only a few components of a high-dimensional spatially distributed field parameter while under measurement attacks. The problem of each agent only estimating a few components of the field was previously considered in reference~\cite{SahuRandomFields} with static fields and reference~\cite{KhanKF} that estimated time-varying random fields by designing distributed Kalman Filters, but neither considered adversarial attacks on the agents. {\color{black} This paper studies distributed field estimation with attacks on the agents.}

We consider a setup that is similar to the setup of our previous work~\cite{ChenSAGE}: a team of agents each makes noisy measurements (over time) of a fraction of an unknown field parameter, and an adversary arbitrarily manipulates some of these measurements. The goal in this paper, however, is different than the goal in~\cite{ChenSAGE}. In~\cite{ChenSAGE}, each agent attempts to estimate the \textit{entire} unknown parameter, while, here, each agent attempts, in collaboration with nearby neighbors, to estimate only a portion of the unknown parameter in which it is interested. Hence, the setup here is more practical than in~\cite{ChenSAGE} when fields are spatially very large (for example, in monitoring traffic conditions over a city~\cite{TrafficRSU} or in field estimation in the power grid~\cite{PowerGrid}). 

Because each agent only estimates a fraction of the field, there are major difficulties not addressed in~\cite{ChenSAGE} that we successfully consider here. Namely, since neighboring agents may be interested in \textit{different} portions of the field, each agent must additionally process the data received from its neighbors to extract the information relevant to  the portion of the field that it wishes to estimate. {\color{black} This leads to several technical differences between our work here and~\cite{ChenSAGE}. First, the heterogenous processing goals induce different topology conditions on the communication network. Second, the notion of consensus needs to be appropriately modified, as agents are no longer interested in reaching an exact consensus on the entire parameter. Rather, the focus is to attain consensus on overlapping components of interest. This changes the nature of the dynamics and interactions between the agents and requires new analysis methods that we develop in this paper.}

This paper describes \textbf{SAFE}, the Saturating Adaptive Fields Estimator, a distributed field estimator that is resilient to measurement attacks. \textbf{SAFE} is a \textit{consensus+innovations} estimator~\cite{Kar1, Kar3} where each agent iteratively updates a local estimate as a weighted sum of its previous estimate, its neighbors' estimates, and its local innovation, the difference between the agent's observed measurement and predicted measurement (based on its estimate). In the estimate update, each agent applies an adaptive gain to its local innovation to ensure that its magnitude is below a time-varying threshold. A key challenge for \textbf{SAFE} is designing this threshold to mitigate the effect of compromised measurements without limiting the information from the uncompromised measurements. The performance of \textbf{SAFE} critically depends on properly selecting the weights and threshold used in the estimate update. We provide a procedure to select this properly, and we show that, under sufficient conditions on the compromised measurement streams and the relationship between the physical coupling and the cyber communication network, \textbf{SAFE} guarantees that all of the agents' estimates are strongly consistent, i.e., each agent's estimate converges to the true value of the portion of the field that it seeks to estimate. To prove consistency, we first show that, for each component of the field, all interested agents reach consensus in their estimates, and then we show that the value of the estimate on which the agents agree converges almost surely to the true value of the field.

The rest of this paper is organized as follows. In Section~\ref{sect: background}, we provide background on the measurement, communication, and attack models, and we formalize the resilient distributed field estimation task. Section~\ref{sect: algorithm} presents \textbf{SAFE}, a distributed field estimator that is resilient to measurement attacks. In Section~\ref{sect: analysis}, we show that, under sufficient conditions on the compromised measurement streams, the physical coupling of the parameter to the measurement streams, and the connectivity of the cyber communication network, \textbf{SAFE} ensures strongly consistent local estimates. Section~\ref{sect: examples} illustrates the performance of \textbf{SAFE} through numerical examples, and we conclude in Section~\ref{sect: conclusion}. 

\section{Background}\label{sect: background}
This section reviews notation and background on the measurement, communication, and attack models, and formalizes the distributed field estimation task.

\subsection{Notation}
Let $\mathbb{R}^k$ be the $k$ dimensional Euclidean space, $I_k$ the $k$ by $k$ identity matrix, and $\mathbf{1}_k$ the $k$ dimensional vector of ones. The column vector $e_j$, $j = 1, 2, \dots, k$, is the $j^{\text{th}}$ canonical basis vector of $\mathbb{R}^k$: $e_j$ is a column vector with $1$ in the $j^{\text{th}}$ row and $0$ everywhere else. The operators $\left\lVert \cdot \right \rVert_2$ and $\left\lVert \cdot \right\rVert_\infty$ are, respectively, the $\ell_2$ norm and the $\ell_\infty$ norm. For matrices $A$ and $B$, $A \otimes B$ is the Kronecker product. For a symmetric matrix $A = A^\intercal$, $A \succeq 0$ ($A \succ 0$) means that $A$ is positive semidefinite (positive definite). For a matrix $A$, $[A]_{i, j}$ is the element in its $i^{\text{th}}$ row and $j^{\text{th}}$. For a vector $v$, $[v]_i$ is its $i^{\text{th}}$ component.

Let $G = (V, E)$ be a simple undirected graph (no self loops or multiple edges) with vertex set $V = \left\{1, \dots, N \right\}$ and edge set $E$. The neighborhood of vertex $n$, $\Omega_n$, is the set of vertices that share an edge with $n$, and $d_n$ is the size of its neighborhood $\left\lvert \Omega_n \right \rvert$. For the graph $G$, let $D = \diag (d_1, \dots, d_N)$ be the degree matrix, $A$, where $\left[A\right]_{n,l} = 1$ if $\left(n, l\right) \in E$ and $\left[A\right]_{n,l} = 0$ otherwise, be the adjacency matrix, and $L = D - A$ be the Laplacian matrix. The Laplacian matrix $L$ has ordered eigenvalues
$ 0 = \lambda_1(L) \leq \cdots \leq \lambda_N(L),$ and eigenvector $\mathbf{1}_N$ associated with the eigenvalue $\lambda_1(L) = 0.$ For a connected graph $G$, $\lambda_2(L) > 0$. References~\cite{Spectral, ModernGraph} provide additional details on spectral graph theory. 

In this paper, all random objects are defined on a common probability space $\left( \Omega, \mathcal{F} \right)$ equipped with a filtration $\mathcal{F}_t$. Reference~\cite{Stochastic} reviews stochastic processes and filtrations. 
The operators $\mathbb{P} \left( \cdot \right)$ and $\mathbb{E} \left( \cdot \right)$ are the probability and expectation operators, respectively. All inequalities involving random variables hold almost surely (a.s.), i.e., with probability $1$, unless otherwise stated.

\subsection{Measurement Model}
Consider a field over a large physical area represented by the unknown, static (field) parameter $\theta^* \in \mathbb{R}^M$. 
A network of $N$ agents or devices $\{1, 2, \dots, N\}$ makes local streams of measurements of the field. 
At each time $t = 0, 1, 2, \dots$, every agent $n$ makes a local measurement
\begin{equation}\label{eqn: noAttackMeasurement}
	y_n(t) = H_n \theta^* + w_n(t),
\end{equation}
where $w_n(t)$ is the local measurement noise. The measurement $y_n(t)$ has dimension $P_n$, i.e., at each time $t$, it makes $P_n$ \textit{scalar} measurements, with $P_n << M$, where $M$ is the dimension of the field parameter $\theta^*$. 
\begin{assumption}\label{ass: noise} 
The measurement noise $w_n(t)$ is temporally independently and identically distributed (i.i.d.) with zero mean and covariance matrix $\Sigma_n$. The measurement noise is independent across agents. The sequence $\left\{ w_n(t) \right\}$ is $\mathcal{F}_{t+1}$ adapted and independent of $\mathcal{F}_t$.
\end{assumption}  
The filtration $\mathcal{F}_t$ in Assumption~\ref{ass: noise} will be defined shortly. We use the same measurement model as~\cite{ChenSAGE}, but, as we will detail in~\ref{sect: fieldEstimation}, this paper studies a different estimation problem. In~\cite{ChenSAGE}, every agent attempts to estimate the entire parameter $\theta^*$, while, here, each agent attempts to estimate only a few components of $\theta^*$.

{\color{black}The matrix $H_n$ is the local measurement matrix of agent $n$ and models the physical coupling between the field $\theta^*$ and the local measurements $y_n(t)$. It states which components of the field affect the measurement $y_n(t)$. Each agent knows its measurement matrix $H_n$ a priori. We assume $H_n$ is sparse, which means that the measurement~\eqref{eqn: noAttackMeasurement} provides information on only a few components of $\theta^*$, i.e., each agent's measurement streams are only coupled to a few components of the (high-dimensional) field parameter. For example, in a robotic mapping and navigation application, where a team of robots attempts to estimate the locations of obstacles, each robot's local measurements only depend on its surroundings. If each component of $\theta^*$ represents the state of a particular location (e.g., if a location is occupied by an obstacle), then, each robot only measures its local surroundings, and the local measurements are physically coupled to only a few components of $\theta^*$, corresponding to nearby locations. }
The measurement matrix $H_n \in \mathbb{R}^{P_n \times M}$ captures the \textit{physical coupling} between the measurements $y_n(t)$ and the field $\theta^*$. The local measurement streams of agent $n$ are physically related to the $m^{\text{th}}$ component of $\theta^*$ if one or more entries of the $m^{\text{th}}$ \textit{column} of $H_n$ are nonzero. For each agent $n = 1, \dots, N$, we formally define the physical coupling set $\widetilde{\mathcal{I}}_n$ as follows.

\begin{definition}[Physical Coupling Set]
	For each agent $n$, the physical coupling set $\widetilde{\mathcal{I}}_n$ is the set of all components of $\theta^*$ that affect its local measurement $y_n(t)$. Formally, $\widetilde{\mathcal{I}}_n$ is the set of indices corresponding to nonzero columns of $H_n$, i.e.,
\begin{equation}\label{eqn: auxiliaryInterest}
	\widetilde{\mathcal{I}}_n = \left\{ m \in \left\{1, \dots M \right\} \vert H_n e_m \neq 0 \right\},
\end{equation}
where $e_m$ is the $m^{\text{th}}$ canonical (column) vector of $\mathbb{R}^M$. 
\end{definition}

\noindent Through its measurement $y_n(t)$, agent $n$ collects information about the components of $\theta^*$ specified by its physical coupling set.

We now describe the indexing convention from~\cite{ChenSAGE} for labeling all of the agents' measurements. Let $\mathbf{y}_t$ stack all of the agents' measurements at time $t$, i.e.,
\begin{equation}\label{eqn: stackedMeasurement}
	\mathbf{y}_t = \left[ \begin{array}{ccc} y_1^\intercal (t) & \cdots & y_N^\intercal (t) \end{array} \right]^\intercal = \mathcal{H} \theta^* + \mathbf{w}_t, 
\end{equation}
where 
	$\mathbf{w}_t = \left[\begin{array}{ccc} w_1^\intercal (t) & \cdots & w_N^\intercal(t) \end{array} \right]^\intercal$ and $\mathcal{H} = \left[\begin{array}{ccc} H_1^\intercal & \cdots & H_N^\intercal \end{array} \right]^\intercal$
are the stacked measurement noises and measurement matrices ($H_1, \dots, H_N$), respectively. The vector $\mathbf{y}_t$ has dimension $P = \sum_{n = 1}^N P_n$. 
In~\eqref{eqn: stackedMeasurement}, we represent $\mathbf{y}_t$ as the stack or concatenation of the \textit{vectors} $y_1(t), \dots, y_N(t)$. We now represent $\mathbf{y}_t$ as the concatenation of \textit{scalars}, labeling each scalar component of $\mathbf{y}_t$ from $1$ to $P$:
\begin{equation} \label{eqn: yRows}
	\mathbf{y}_t = \left[ \begin{array}{ccc} \underbrace{\begin{array}{ccc} y^{(1)}(t) & \cdots & y^{(P_1)}(t) \end{array}}_{y_1^\intercal(t)} & \cdots & \underbrace{\begin{array}{ccc}y^{(\overline{P}_N+1)} (t) & \cdots & y^{(P)}(t)\end{array}}_{y_N^\intercal(t)}  \end{array} \right]^\intercal,
\end{equation}
where  {\color{black} $\overline{P}_0 = 0$ and $\overline{P}_n = \sum_{j=1}^{n-1} P_j$ for $n = 2, \dots, N$. In~\eqref{eqn: yRows}, each element $y^{(p)} (t)$, $p = 1, \dots, P$, is a \textit{scalar}.\footnote{The variable $p$ indexes the scalar components.}} The scalar components $y^{(1)}(t), \dots, y^{(P_1)}(t)$ make up the vector measurement $y_1(t)$ for agent $1$ at time $t$, and, in general, the vector measurement of agent $n$ at time $t$ is made up of the scalar components
\begin{equation}\label{eqn: indexConvention}
	y_n(t) = \left[\begin{array}{ccc} y^{(\overline{P}_n + 1)}(t) & \cdots & y^{(\overline{P}_n + P_n )} (t) \end{array} \right]^\intercal.
\end{equation}
Following~\cite{ChenSAGE}, we now define measurement streams as follows.
\begin{definition}[Measurement Stream~\cite{ChenSAGE}]
A measurement stream 
is the collection of the scalar measurement $y^{(p)}(t)$ over all times $t = 0, 1, 2, \dots$, {\color{black} where $p$ is the index assigned in~\eqref{eqn: stackedMeasurement}.}\hfill $\small \blacksquare$
\end{definition}
\noindent For the rest of this paper, we refer to a measurement stream by its component index $p$.  The set $\mathcal{P} = \left\{1, \dots, P \right\}$ is the collection of all measurement streams. Following the convention of~\eqref{eqn: yRows} and~\eqref{eqn: indexConvention}, we also label the rows of $\mathcal{H}$ from $1$ to $P$, i.e., $\mathcal{H} = \left[\begin{array}{ccc} h_1 & \cdots & h_P \end{array} \right]^\intercal$, subject to the following assumption: 
\begin{assumption}\label{ass: normalization}
	The measurement vectors $h_p$, $p = 1, \dots, P$, have unit $\ell_2$ norm, i.e., $\left\lVert h_p \right\rVert_2 = 1$.
\end{assumption}

 

\subsection{Distributed Field Estimation}\label{sect: fieldEstimation}
Up to here, the problem set up is similar to that of~\cite{ChenSAGE}, where every agent attempts to estimate the \textit{entire} parameter $\theta^*$. In this paper, unlike in~\cite{ChenSAGE}, the goal of each agent $n$ is to estimate a subset of the components of $\theta^*$ as specified by its \textit{interest set} $\mathcal{I}_n$. 
\begin{definition}[Interest Set]\label{def: interestSet}
For each agent $n$, the interest set $\mathcal{I}_n$ is the ordered set of indices corresponding to the components of $\theta^*$ that it wishes to estimate. 
\end{definition}
\noindent Each interest set $\mathcal{I}_n$ is a subset of $\left\{1, \dots, M \right\}$, where $M$ is the dimension of $\theta^*$. We now review the indexing convention for interest sets from~\cite{SahuRandomFields}. {\color{black} For $r = 1, \dots, \left\lvert \mathcal{I}_n \right\rvert$, the expression
$\mathcal{I}_n(r) = m$
means that the $r^{\text{th}}$ element of the interest set $\mathcal{I}_n$ is component index $m$, and the expression
$\mathcal{I}_n^{-1} (m) = r$ means that component index $m$ is the $r^{\text{th}}$  element of the interest set $\mathcal{I}_n$. The elements of each interest set $\mathcal{I}_n$ are in ascending order, i.e., $\mathcal{I}_n(1) < \mathcal{I}_n(2) < \cdots < \mathcal{I}_n\left( \left\lvert \mathcal{I}_n \right\rvert \right)$.}  
Following~\cite{SahuRandomFields}, we assume that, for each agent, the physical coupling set is a subset of the interest set:
\begin{assumption}\label{ass: interestSubset}
	The physical coupling set $\widetilde{\mathcal{I}_n}$ is a subset of the interest set $\mathcal{I}_n$. 
\end{assumption}
\noindent Assumption~\ref{ass: interestSubset} states that each agent $n$ is at least interested in estimating all components of the parameter that are coupled to its local measurement streams. {\color{black} There is a relationship between the measurement matrix $H_n$ and the interest set $\mathcal{I}_n$. As a consequence of Assumption~\ref{ass: interestSubset}, each interest set $\mathcal{I}_n$ must include the indices corresponding to the nonzero columns $H_n$ (see~\eqref{eqn: auxiliaryInterest}). In this paper, the interest set $\mathcal{I}_n$ need not depend on the inter-agent communication network.}

In addition to the interest set $\mathcal{I}_n$, which collects all of the components in which agent $n$ is interested, we also need to keep track of all agents interested in a particular component. For each $m = 1, \dots, M$, define the set
\begin{equation}\label{eqn: Jn}
	\mathcal{J}_m = \left\{n \in \left\{1, \dots, N\right\} \vert m \in \mathcal{I}_n \right\}. 
\end{equation}
as the set of all \textit{agents} that are interested in estimating the component $m$ of the field~$\theta^*$. 
\begin{assumption}\label{ass: interestSubset2}
	For all $m = 1, \dots, M$, $\left\lvert \mathcal{J}_m \right\rvert > 0$. 
\end{assumption}
\noindent {\color{black}Assumption~\ref{ass: interestSubset2} states that for each component of $\theta^*$, there is at least one interested agent.  There does \textit{not} need to be a single agent interested in estimating all components of $\theta^*$.}

An important concept in estimation tasks is global observability, that we define next in context of parameter estimation.
\begin{definition}[Global Observability]\label{def: globalObs}
	Consider a set of measurement streams $\mathcal{X} = \left\{p_1, \dots, p_{\left\lvert \mathcal{X} \right\rvert} \right\} \subseteq \mathcal{P}$, and let 
	$\mathcal{H}_{\mathcal{X}} = \left[\begin{array}{ccc} h_{p_1} & \cdots & h_{p_{\left\lvert \mathcal{X} \right \rvert}} \end{array} \right]^\intercal$
be the matrix that stacks all measurement vectors associated with streams in $\mathcal{X}$. 
The set $\mathcal{X}$ is globally observable if the $M \times M$ observability Grammian
	$\mathcal{G}_{\mathcal{X}} = \mathcal{H}_{\mathcal{X}}^\intercal \mathcal{H}_{\mathcal{X}} = \sum_{p \in \mathcal{X}} h_p h_p^\intercal$
is invertible. \hfill $\small \blacksquare$
\end{definition}
\noindent If a set of measurement streams is globally observable, then, when there is no measurement noise, it is possible to unambiguously determine $\theta^*$ using a single snapshot in time of the measurements in $\mathcal{X}$. 

\begin{assumption}\label{ass: globalObs}
	The set of all measurement streams $\mathcal{P}$ is globally observable.
\end{assumption}

\noindent Even a centralized estimator, which has access to all measurement streams, requires {\color{black} global observabilty} to construct a consistent estimate. So, it makes sense to assume it in our distributed setting. {\color{black}Assumption~\ref{ass: globalObs} is standard in previous work on distributed estimation~\cite{Kar2, Kar3, SahuRandomFields}.}

{\color{black} While we assume \textit{global} observabilty, we do not require \textit{local} observability. }An individual agent most likely does not have enough information from its local measurements alone to estimate all of the components in its interest set. {\color{black} For example, in robotic mapping and navigation, a robot may be interested in estimating the state of a location that it does not directly measure.}
To accomplish their estimation goals, agents need to communicate with neighbors. {\color{black} We provide an example in Section~\ref{sect: examples} where agents \textit{must} communicate with neighbors to accomplish their estimation goals.} They do this over a \textit{cyber} communication network, modeled by a time varying graph $G(t) = (V, E(t))$, where $V$ is the set of all agents, and $E(t)$ is the set of inter-agent communication links at time $t$. For each agent $n$, $\Omega_n(t)$ is its neighborhood at time $t$. In real world conditions, their local cyber networks change, for example, due to shadowing or random failures of local wireless channels. 

\begin{assumption}\label{ass: iidGraphs}
The Laplacians $\left\{L(t)\right\}$ (of the graphs $\left\{ G(t) \right\}$) form an i.i.d. sequence with mean $\mathbb{E} \left[L(t) \right] = \overline{L}.$ The sequence $\left\{L(t)\right\}$ is $\mathcal{F}_{t+1}$-adapted and independent of $\mathcal{F}_t$. The sequence of Laplacians is independent of $\left\{w_n(t) \right\}$. 
\end{assumption}

\noindent In our setup, the only sources of randomness are the sequences $\left\{w_n(t)\right\}$, $n = 1, \dots, N$ and {\color{black} $\left\{L(t) \right\}$.} The filtration $\mathcal{F}_t$ is the natural filtration:
\begin{equation}\label{eqn: naturalFiltration}
	\mathcal{F}_t = \sigma\left(L_0, \dots, L_{t-1}, w_1(0), \dots, w_1(t-1), \dots, w_N(0), \dots, w_N(t-1) \right)
\end{equation}

We make the following assumption about the relationship between the cyber-layer communication network $G(t)$ and the physical coupling between the components of $\theta^*$ and the agents' local measurements. For each parameter component $m = 1, \dots, M$, let $G_m(t)$ be the subgraph of $G(t)$ induced by $\mathcal{J}_m$, the agents that are interested in the $m^{\text{th}}$ component.\footnote{$G_m(t)$ is the graph formed from $G$ by the subset of agents (vertices) interested in the $m^{\text{th}}$ component ($\mathcal{J}_m$) and the subset of edges connecting pairs of agents in $\mathcal{J}_m$. }  Recall that, as a consequence of Assumption~\ref{ass: interestSubset}, the set of agents interested in a specific component of $\theta^*$ includes the set of agents whose measurements are physically coupled to that specific component. 

\begin{assumption}\label{ass: connectivity}
	{\color{black}For each component $m = 1, \dots, M$ of $\theta^*$, let $L_m(t)$ be the graph Laplacian of the induced subgraph $G_m(t)$, and let $\overline{L}_m = \mathbb{E} \left[L_m(t) \right]$. Then, $\overline{L}_m$ satsifies  
\begin{equation}\label{eqn: connectivityCondition}
	\lambda_2 (\overline{L}_m) > 0,
\end{equation} 
for each $m = 1, \dots, M$.}
\end{assumption} 
\noindent {\color{black} The communication network induced by $\mathcal{J}_m$ is connected on average for every $m = 1, \dots, M$. That is, the mean Laplacian $\overline{L}_m$ satisifies the connectivity condition~\eqref{eqn: connectivityCondition}. } Assumption~\ref{ass: connectivity} is a sufficient condition for consistent distributed field estimation in the absence of adversarial attacks~\cite{SahuRandomFields}. We assume it also in this paper for distributed field estimation in the presence of adversaries. 

\subsection{Attack Model}
A malicious attacker corrupts a subset of the measurement streams, replacing the true measurements with arbitrary values, modeled as
\begin{equation}\label{eqn: attackMeasurement}
	y_n(t) = H_n \theta^* + w_n(t) + a_n(t).
\end{equation}
{\color{black} The attacks may be either deterministic or random, and the agents are unaware, a priori, of how the attacker chooses $a_n(t)$. We seek to ensure that all agents consistently estimate their components of interest regardless of how the attacked is carried out. }Using the same indexing convention as~\eqref{eqn: indexConvention}, we label each \textit{scalar} component of $a_n(t)$ with the indices $\overline{P}_n + 1, \dots, \overline{P}_n + P_n,$ where, recall $\overline{P}_n = \sum_{j = 1}^{n-1} P_n$. That is, $a_n(t) = \left[\begin{array}{ccc} a^{(\overline{P}_n + 1)}(t) & \cdots & a^{(\overline{P}_n + P_n)}(t) \end{array} \right]^\intercal.$
We consider measurement stream $p$ ($p = 1, \dots, P$) to be compromised or under attack if, at \textit{any} time $t = 0, 1, \dots,$ the scalar attack element $a^{(p)}(t) \neq 0$. The set of all measurement streams $\mathcal{P}$ may be partitioned into a set of compromised measurement streams \begin{equation} \mathcal{A} = \left\{p \in \mathcal{P} \Big\vert \exists t = 0, 1, \dots, \: a^{(p)}(t) \neq 0 \right\}, \end{equation} and a set of uncompromised measurement streams $\mathcal{N} = \mathcal{P} \setminus \mathcal{A}$. These sets do not change over time, and agents do not know which measurement streams are compromised. The attacker satisfies the following assumption:
\begin{assumption}\label{ass: attackSetSize}
	The attacker may only manipulate a subset (but not all) of the measurement streams, i.e., $0 \leq \left\lvert \mathcal{A} \right\rvert < P$. 
\end{assumption}

\section{\textbf{SAFE}: Saturating Adaptive Field Estimator}\label{sect: algorithm}
In this section, we describe \textbf{SAFE}, the Saturating Adaptive Field Estimator, a distributed field estimator that is resilient to measurement attacks. \textbf{SAFE} is a \textit{consensus + innovations} estimator~\cite{Kar1, Kar3}, and it is the resilient version of the distributed field estimator from~\cite{SahuRandomFields}. That is, \textbf{SAFE} copes with measurement attacks while the estimator from~\cite{SahuRandomFields} does not. \textbf{SAFE} is based on the Saturating Adaptive Gain Estimator (\textbf{SAGE}) for resilient distributed \textit{parameter} estimation we presented in~\cite{ChenSAGE}. In~\cite{ChenSAGE}, every agent is intersted in estimating the entire parameter $\theta^*$ while under measurement attacks. Unlike~\cite{ChenSAGE} (\textbf{SAGE}), this paper (\textbf{SAFE}) addresses resilient distributed \textit{field} estimation, where each agent is only interested in estimating a subset of the components of $\theta^*$. 

\subsection{Algorithm}\label{sect: algorithmDescription}
In \textbf{SAFE}, each agent $n$ maintains a $\left\lvert \mathcal{I}_n \right\rvert$-dimensional local estimate $x_n(t)  \in \mathbb{R}^{\left\lvert \mathcal{I}_n \right\vert}$ of all of the components of $\theta^*$ which it is interested in estimating, where the $i^{\text{th}}$ component of $x_n(t)$ is an estimate of the $i^{\text{th}}$ member of the interest set $\mathcal{I}_n$. 
Each agent initializes its estimate as $x_n(0) = 0$. Then each iteration of \textbf{SAFE} consists of three steps: 1.) message passing, 2.) message processing, and 3.) estimate update. 

\underline{Message Passing}: Each agent $n$ communicates its current estimate $x_n(t)$ to each of its neighbors $l \in \Omega_n(t)$. 

\underline{Message Processing}: Neighboring agents may be interested in different components of $\theta^*$. Each agent $n$ extracts only the information, from its neighbors' estimates, about the components of $\theta^*$ it is interested in estimating and ignores the irrelevant information. The agents use the following message processing procedure.
Each agent $n$ receives $x_l(t)$ from each of its neighbors $l \in \Omega_n(t)$. Then, agent $n$ constructs a \textit{censored} version, $x^c_{l, n}(t)$, of the estimate from agent $l$ element-wise as follows:
\begin{equation}\label{eqn: censorReceived}
	\left[x^c_{l, n}(t) \right]_i = \left\{ \begin{array}{ll} \left[x_{l}(t)\right]_{\mathcal{I}_l^{-1} \left(\mathcal{I}_n(i)\right)}, & \text{if } \mathcal{I}_n(i) \in \mathcal{I}_l, \\ 0, &\text{otherwise.} \end{array} \right.
\end{equation}
Agent $n$ constructs the censored message $x^c_{l, n}(t)$ element-wise by placing the $j^{\text{th}}$ element of $x_l(t)$, where $j = \mathcal{I}_l^{-1} \left(\mathcal{I}_n(i)\right)$ (i.e., agent $l$'s estimate of $\left[\theta^* \right]_{{\mathcal{I}}_n(i)}$, following the indexing convention for interest sets introduced in Section~\ref{sect: fieldEstimation}), in the $i^{\text{th}}$ element of $x^c_{l, n}(t)$ if agent $l$ is also interested in $\mathcal{I}_n (i)$, and placing zero in the $i^{\text{th}}$ element otherwise. 

In addition to processing its neighbors' estimates, each agent $n$ also needs to compute a processed version of its own estimate for the estimate update step. For each of its neighbors $l \in \Omega_n(t)$, agent $n$ constructs $x^p_{l, n}$, a processed version of its own estimate, element-wise as follows:
\begin{equation}\label{eqn: censorSelf}
	\left[x^p_{l, n}(t) \right]_i = \left\{ \begin{array}{ll} \left[x_{n}(t)\right]_i, & \text{if } \mathcal{I}_n(i) \in \mathcal{I}_l, \\ 0, &\text{otherwise.} \end{array} \right.
\end{equation}
That is, $x_{l, n}^p (t)$ preserves the $i^{\text{th}}$ element of $x_n(t)$ if agent $l$ is also interested in $\mathcal{I}_n(i)$ and places zero in the other elements. 

\underline{Estimate Update}: Each agent maintains a time-averaged measurement:
\begin{equation}\label{eqn: timeAvgMeasurement}
	\overline{y}_n(t) = \frac{t}{t+1} \overline{y}_n(t-1) + \frac{1}{t+1} y_n(t),
\end{equation}
with initial condition $\overline{y}_n(-1) = 0$. If measurement stream $p$ is uncompromised ($p \in \mathcal{N}$), then, the time averaged version of $p$ follows
	$\overline{y}^{(p)} = h_p^\intercal \theta^* + \overline{w}^{(p)}(t),$
where $\overline{w}^{(p)}(t) = \frac{1}{t+1} \sum_{j = 0}^t w^{(p)}(j).$ In addition to the time-averaged measurement $\overline{y}_n(t)$, the estimate update step requires the censored measurement matrix $H^c_n$, which is the measurement matrix $H_n$ after removing all \textit{columns} whose indices are not in $\mathcal{I}_n$. The matrix $H^c_n$ has $P_n$ rows and $\left\lvert \mathcal{I}_n \right \rvert$ columns. 

Each agent updates its estimate as
\begin{equation}\label{eqn: estimateUpdate}
\begin{split}
	\!x_n(t+1)\! = &x_n(t) - \beta_t \!\! \!\!\sum_{l \in \Omega_n(t)} \!\! \underbrace{\left(x_{l, n}^p(t) - x_{l, n}^c(t)\right)}_{\text{Neighborhood Consensus}}\!\!\! + \alpha_t {H_n^c}^\intercal K_n(t) \underbrace{\left( \overline{y}_n(t) - H_n^c x_n(t)\right)}_{\text{Local Innovation}},\!\!
\end{split}
\end{equation}
where the innovations and consensus weight sequences $\alpha_t > 0$ and $\beta_t > 0$ will be defined shortly, and $K_n(t)$ is a diagonal gain matrix that depends on $\overline{y}_n (t) - H_n^c x_n(t)$. Recall that, following~\eqref{eqn: indexConvention},  $\overline{P}_n + 1, \dots, \overline{P}_n + P_n$ are the indices for the components of $\overline{y}_n(t)$. For each $p = \overline{P}_n + 1, \dots, \overline{P}_n + P_n$, define the scalar gain
\begin{equation}\label{eqn: smallK}
	k_p (t) = \min\left(1, \gamma_t \left\lvert y^{(p)}(t) - {h^c_p}^\intercal x_n(t) \right \rvert^{-1} \right),
\end{equation}
where ${h^c_p}^\intercal$ is the row vector ${h_p^\intercal}$ after removing components whose indices are not in $\mathcal{I}_n$, and $\gamma_t$ is a scalar threshold sequence that will be defined shortly. Then, $K_n(t)$ is defined as
\begin{equation}\label{eqn: KnDef}
	K_n(t) = \diag \left(k_{\overline{P}_n + 1}(t), \dots, k_{\overline{P}_n + P_n}(t) \right).
\end{equation}

The gain matrix $K_n(t)$ clips (or saturates) each component of the local innovation at the threshold level $\gamma_t$, i.e., it ensures that the $\ell_\infty$ norm of the scaled innovation, $\left\lVert K_n(t) \left(\overline{y}_n(t) - H_n^c x_n(t) \right) \right\rVert_\infty$, does not exceed the threshold $\gamma_t$. The goal in clipping the local innovations is to limit the impact of measurement attacks on the estimate update. The key challenge in designing the \textbf{SAFE} algorithm is selecting the threshold sequence. If the threshold $\gamma_t$ is not chosen correctly, we may limit the impact of uncompromised measurement streams ($p \in \mathcal{N}$) on the estimate update. We choose the threshold $\gamma_t$, along with the innovations and consensus weight sequences $\alpha_t$ and $\beta_t$, to balance these two effects as follows:
\begin{enumerate}
	\item The sequences $\left\{\alpha_t\right\}$ and $\left\{\beta_t \right\}$ are given by
		$\alpha_t = \frac{a}{(t+1)^{\tau_1}}, \: \beta_t = \frac{b}{(t+1)^{\tau_2}},$
	where $a, b > 0$ and $0 < \tau_2 < \tau_1 < 1$.
	\item The threshold sequence $\left\{ \gamma_t \right\}$ is given by
		$\gamma_t = \frac{\Gamma}{(t+1)^{\tau_\gamma}},$
	where $\Gamma > 0$ and $0 < \tau_\gamma < \min \left(\frac{1}{2}, \tau_1 - \tau_2, 1-\tau_1 \right)$.
\end{enumerate}

The innovations and consensus weight sequences decay in time, with the innovation weight sequence $\left\{ \alpha_t \right\}$ decaying at a faster rate. The threshold sequence $\left\{\gamma_t \right\}$ also decays over time. Intuitively, the threshold $\gamma_t$ decaying means that the estimate update~\eqref{eqn: estimateUpdate} allows initially for large contributions from the local innovation, incorporating more information from the measurements. Then, as the number of iterations increases, the local estimates should move closer to the true value of the parameter, resulting in smaller magnitude local innovations for uncompromised measurement streams. Accordingly, decreasing the threshold $\gamma_t$ over time, we expect the update~\eqref{eqn: estimateUpdate} to limit the impact of compromised measurements while still incorporating enough information from the uncompromised measurements. The weight and threshold selection is critical to the performance and resilience of \textbf{SAFE}. If these are selected improperly, then, the performance guarantees, stated below, no longer hold. 

In the distributed field estimation problem, each agent is only interested in estimating a few components of the parameter. This differs from the more common distributed parameter estimation problem, like the setup of \textbf{SAGE} in~\cite{ChenSAGE}, where every agent is interested in estimating all components of the parameter. To account for the different interest sets between neighboring agents, \textbf{SAFE} introduces a message processing step (equations~\eqref{eqn: censorReceived} and~\eqref{eqn: censorSelf}), not found in \textbf{SAGE}~\cite{ChenSAGE}, and uses the processed messages in the estimate update. By handling different interest sets among agents,~\textbf{SAFE} is more general than~\textbf{SAGE}: if all agents have the same interest set $\mathcal{I}_n = \left\{1, \dots, M\right\}$, then, \textbf{SAFE} reduces to \textbf{SAGE}.

\subsection{Main Result: Performance of \textbf{SAFE}}
We now present our main result. 
Define $\theta^*_{\mathcal{I}_n}$ element-wise as
	$\left[\theta^*_{\mathcal{I}_n} \right]_i = \left[\theta^*\right]_{\mathcal{I}_n(i)},$
for $i = 1, \dots, \left\lvert \mathcal{I}_n \right\rvert$. That is, $\theta^*_{\mathcal{I}_n}$ is the $\left\lvert \mathcal{I}_n \right\rvert$-dimensional vector that collects the components of $\theta^*$ in which agent $n$ is interested. The following theorem establishes the strong consistency of \textbf{SAFE} on the local interest sets.

\begin{theorem}\label{thm: main}
	Let $\mathcal{A}=\left\{p_1, \dots, p_{\left\lvert \mathcal{A} \right\rvert} \right\}$ be the set of compromised measurement streams, and let $\mathcal{H}_{\mathcal{A}} = \left[\begin{array}{ccc} h_{p_1} & \cdots & h_{p_{\left\lvert \mathcal{A} \right\rvert}} \end{array} \right]^\intercal.$ If for the matrix 
	$\mathcal{G}_{\mathcal{N}} = \sum_{p \in \mathcal{N}} h_p h_p^\intercal$,
\begin{equation}\label{eqn: resilienceCondition}
	\lambda_\min \left(\mathcal{G}_{\mathcal{N}} \right) > \Delta_{\mathcal{A}},
\end{equation}
where $\mathcal{N} = \mathcal{P}\setminus \mathcal{A}$ and
	 $\Delta_{\mathcal{A}} =  \max \limits_{v \in \mathbb{R}^{\left\lvert \mathcal{A} \right\rvert}, \left\lVert v \right \rVert_\infty \leq 1} \left\lVert \mathcal{H}_{\mathcal{A}}^\intercal v \right\rVert_2,$
then, for all $n$ and all $0 \leq \tau_0 < \min\left(\tau_\gamma, \frac{1}{2} - \tau_\gamma\right)$, we have
\begin{equation}\label{eqn: localConsistency}
	\mathbb{P} \left( \lim_{t \rightarrow \infty} \left(t+1 \right)^{\tau_0} \left\lVert x_n(t) - \theta^*_{\mathcal{I}_n} \right\rVert_2 = 0 \right) = 1.
\end{equation}

\end{theorem}

Theorem~\ref{thm: main} states that, under the resilience condition in~\eqref{eqn: resilienceCondition}, \textbf{SAFE} ensures that \text{all} of the agents' local estimates converge a.s.~to the true value of the parameter on their respective interest sets. The resilience condition~\eqref{eqn: resilienceCondition} is a condition on the global redundancy of the noncompromised measurement streams. Intuitively, it states that the noncompromised measurement streams have enough redundancy in measuring $\theta^*$ to overcome the impact of the compromised measurement streams. The resilience condition~\eqref{eqn: resilienceCondition} for consistent distributed \textit{field} estimation under measurement attacks is the same resilience condition for consistent distributed \textit{parameter} estimation under measurement attacks~\cite{ChenSAGE}. 
Thus, the \textit{resilience} of our distributed field estimator \textbf{SAFE} does not depend on the agents' local interest sets; it only depends on which measurement streams are compromised. 

The performance of \textbf{SAFE} also depends on the connectivity of the cyber communication graph and the physical coupling between the parameter and the agents' measurements. {\color{black} It is difficult to determine how the local neighborhood, $\Omega_n(t)$, of a single agent, $n$, directly affects the \textit{global} dynamics of the algorithm. Thus, instead of considering the local neighborhood structure $\Omega_n(t)$, we provide a necessary condition (for resilient field estimation) on the \textit{global} topology of the communication network $G(t)$ }  Specifically, following Assumption~\ref{ass: connectivity}, we require that, 
for each $m = 1, \dots, M$ (i.e., each component of the parameter $\theta^*$), the induced subgraph $G_{m} (t)$ be connected on average, {\color{black} i.e., $\lambda_2 \left( \overline{L}_m \right) > 0$ for all $m = 1, \dots, M$.} 

\section{Performance Analysis}\label{sect: analysis}
In this section, we prove Theorem~\ref{thm: main}, which states that, under \textbf{SAFE}, all of the agents' local estimates converge a.s.~to the true value of the parameter on their respective interest sets as long as the resilience condition~\eqref{eqn: resilienceCondition} is satisfied. {\color{black}Our analysis does not depend on the attacker's strategy, i.e., how the attacker chooses $\left\{a_n(t)\right\}$. We carry out our analysis over sample paths, and, in the case of random attacks, we do not make any assumptions about the distribution or statistics of $\left\{a_n(t)\right\}$. }The proof depends on several existing intermediate results from the literature~\cite{ChenDistributed1, ChenSAGE, Schur} that are provided in the Appendix. All inequalities involving random variables hold a.s.~(with probability $1$) unless otherwise stated. 

Our proof requires new techniques not found in existing work on resilient distributed parameter estimation~\cite{ChenSAGE}. In distributed parameter estimation settings, we show that the agents' estimates all converge to the network average estimate, and that the network average estimate converges to the true value of the parameter. This method relies on the fact that all agents share the same estimation goal. In distributed field estimation, each agent is only interested in and keeps track of a few components of the parameter, which complicates the notion of the network average estimate. To account for this complication, we analyze the consensus of the local estimates on component-induced subgraphs $G_m(t)$. That is, for each component of the unknown parameter $m = 1, \dots, M$, we consider the subgraph $G_m(t)$ induced by agents who are interested in estimating $[\theta^*]_m$ and show that their estimates of $[\theta^*]_m$ converge a.s.~to the network average estimate taken over $G_m(t)$. Then, we show that the network average estimate taken over $G_m(t)$ converges a.s.~to the true value of $[\theta^*]_m$. 

\subsection{Auxiliary State Transformation}
For convenience and clarity, before proving Theorem~\ref{thm: main}, we introduce an auxiliary state variable for describing the evolution of the agents' local estimates. We follow the convention provided in~\cite{SahuRandomFields}. For each agent $n$, we define the auxiliary state $\widetilde{x}_n(t) \in \mathbb{R}^M$ as follows (recall that $M$ is the dimension of the parameter $\theta^*$):
\begin{equation}\label{eqn: tildeX}
	\left[ \widetilde{x}_n(t) \right]_i = \left\{ \begin{array}{ll} \left[ x_n(t) \right]_{\mathcal{I}^{-1}_n \left( i \right)}, & \text{if } i \in \mathcal{I}_n, \\ 0, &\text{otherwise,} \end{array} \right.
\end{equation}
for each $i =1, \dots, M$. That is, for each component $i = 1, \dots, M$, if agent $n$ is interested in component $i$, then the $i^{\text{th}}$ component of $\widetilde{x}_n(t)$ is agent $n$'s estimate of $[\theta^*]_i$; otherwise, the $i^{\text{th}}$ component of $\widetilde{x}_n(t)$ is zero. Further, for each agent $n$, define the $M \times M$ diagonal matrix
\begin{equation}\label{eqn: qDef}
	Q_n = \diag \left(q^n_1, \dots, q^n_M\right),
\end{equation}
where, for $i  = 1, \dots, M$, 
\begin{equation}\label{eqn: qiDef}
	q^n_i = \left\{ \begin{array}{ll} 1, & \text{if } i \in \mathcal{I}_n, \\ 0, &\text{otherwise.} \end{array} \right.
\end{equation}

Recall that, for each agent $n$, the auxiliary interest set $\widetilde{\mathcal{I}}_n$ is the set of indices of the nonzero \textit{columns} of $H_n$, i.e., the components of the parameter that are coupled to the measurement streams at agent $n$. Assumption~\ref{ass: interestSubset} states that $\widetilde{\mathcal{I}}_n \subseteq \mathcal{I}_n$. As a result of assumption~\ref{ass: interestSubset} and the definition of $q^n_i$, the following holds for each agent $n$:
	$H_n Q_n = H_n.$
This is because $q_i^n = 1$ if the $i^{\text{th}}$ column of $H_n$ is nonzero. 

From the estimate update rule~\eqref{eqn: estimateUpdate}, the auxiliary state $\widetilde{x}_n(t)$ follows
\begin{equation}\label{eqn: auxiliaryUpdate}
\begin{split}
	\widetilde{x}_n(t+1) \! =& \widetilde{x}_n(t) - \beta_t \!\!\!\! \sum_{l \in \Omega_n(t)} \!\!\!\! Q_n Q_l \! \left( \widetilde{x}_n(t) \! - \! \widetilde{x}_l(t)\right) + \alpha_t H_n^\intercal K_n(t)\left(\overline{y}_n(t) - H_n x_n(t) \right).\!\!
\end{split}
\end{equation}
Let $\widetilde{\mathbf{x}}_t = \left[\begin{array}{ccc} \widetilde{x}_1(t)^\intercal & \cdots & \widetilde{x}_N(t)^\intercal \end{array} \right]^\intercal$ and $\overline{\mathbf{y}}_t = \left[\begin{array}{ccc} \overline{y}_1(t)^\intercal & \cdots & \overline{y}_N(t)^\intercal \end{array} \right]^\intercal$ stack, respectively, all of the agents' auxiliary states and measurements at time $t$. 
To represent the evolution of $\widetilde{\mathbf{x}}_t$, we define the $NM \times NM$ matrix $\mathbf{L}_t$ block-wise. Let $\left[\mathbf{L}_t \right]_{n, l} \in \mathbb{R}^{M \times M}$ be the $(n, l)$-th \textit{sub-block} of the matrix $\mathbf{L}_t$, for $n, l = 1, \dots, N$, defined as
\begin{equation}\label{eqn: ltDef}
	\left[\mathbf{L}_t\right]_{n, l} = \left\{ \begin{array}{ll}-Q_n \sum_{i = 1:i \neq n}^N \left[L(t)\right]_{n, i} Q_i, & \text{if } n = l, \\ \left[L(t) \right]_{n, l} Q_n Q_l, & \text{otherwise.}\end{array} \right.
\end{equation}
In~\eqref{eqn: ltDef}, $\left[L(t) \right]_{n, l} \in \mathbb{R}$ is the $(n, l)$-th \textit{element} of the Laplacian, $L(t)$, of the communication graph $G(t)$. 
From~\eqref{eqn: auxiliaryUpdate}, the stacked auxiliary states evolve as
\begin{equation}\label{eqn: stackedTildeX}
\begin{split}
	\widetilde{\mathbf{x}}_{t+1} &=\widetilde{\mathbf{x}}_t - \beta_t \mathbf{L}_t \widetilde{\mathbf{x}}_t + \alpha_t D_H^\intercal \mathbf{K}_t \left(\overline{\mathbf{y}}_t - D_H \widetilde{\mathbf{x}}_t \right),
\end{split}
\end{equation}
where
\begin{align}
	\mathbf{K}_t &= \blkdiag\left(K_1(t), \dots, K_N(t) \right), \label{eqn: matrixKDef} \\
	D_H &= \blkdiag \left(H_1, \dots, H_N \right).
\end{align}
In the sequel, we use the evolution of the auxiliary states~\eqref{eqn: auxiliaryUpdate}, to prove Theorem~\ref{thm: main}. 

\subsection{Network Convergence}
We now prove that the agents' estimates converge component-wise to the network average over the induced subgraphs $G_m(t)$, $m = 1, \dots, M$. Define the $M \times M$ matrix (where, recall, $M$ is the dimension of the parameter $\theta^*$)
\begin{equation}\label{eqn: calDDef}
	\mathcal{D} = \diag\left(\left\lvert{\mathcal{J}_1}\right\rvert^{-1}, \dots, \left\lvert{\mathcal{J}_M}\right\rvert^{-1}\right). 
\end{equation}
Then, we define the $M$-dimensional \textit{generalized network average estimate} as
\begin{equation}\label{eqn: generalizedAverage}
	\overline{\mathbf{x}}_t = \mathcal{D} \left( \mathbf{1}_N^\intercal \otimes I_M \right) \widetilde{\mathbf{x}}_t.
\end{equation}
Note that each component $m = 1, \dots, M$ of $\overline{\mathbf{x}}_t$ follows
	$\left[\overline{\mathbf{x}}_t\right]_m = \frac{1}{\left\lvert \mathcal{J}_m \right\rvert} \sum_{n = 1}^N \left[\widetilde{x}_n(t) \right]_m =  \frac{1}{\left\lvert \mathcal{J}_m \right\rvert} \sum_{n \in \mathcal{J}_m} \left[\widetilde{x}_n(t) \right]_m,$
since, if $n \notin \mathcal{J}_m$ (i.e., if agent $n$ is not interested in $m$), then, by~\eqref{eqn: tildeX}, $\left[\widetilde{x}_n(t) \right]_m = 0$. That is, each component $m$ of $\overline{\mathbf{x}}_t$ is the average estimate of all agents interested in estimating $\left[\theta^* \right]_m$. 

We are interested in analyzing
	$\widehat{x}_n(t) = \widetilde{x}_n(t) - \mathbf{\overline{x}}_t,$
the difference between each agent's auxiliary state and the generalized network average estimate. For each agent $n$, we only wish to consider the components of $\widehat{x}_n(t)$ that belong to the interest set of agent $n$. Thus, instead of analyzing $\widehat{x}_n(t)$ directly, we analyze the behavior of $Q_n \widehat{x}_n(t)$, where $Q_n$ follows~\eqref{eqn: qDef}. Let
\begin{equation}\label{xHatStack}
	\widehat{\mathbf{x}}_t = \widetilde{\mathbf{x}}_t - \left(\mathbf{1}_N \otimes I_M \right) \mathbf{\overline{x}}_t,
\end{equation}
stack $\widehat{x}_n(t)$ across all agents, and define the matrix
\begin{equation}\label{eqn: calQDef}
	\mathcal{Q} = \blkdiag\left(Q_1, \dots, Q_N \right).
\end{equation}
The following lemma describes the behavior of $\mathcal{Q} \widehat{\mathbf{x}}_t$. 
\begin{lemma}\label{lem: resilientConsensus}
	Under \textbf{SAFE}, for every $0 \leq \tau_3 < \tau_\gamma + \tau_1 - \tau_2$, $\mathcal{Q} \widehat{\mathbf{x}}_t$ satisfies
\begin{equation}\label{eqn: resilientConsensus}
	\mathbb{P}\left( \lim_{t \rightarrow \infty} \left(t+1\right)^{\tau_3} \left\lVert \mathcal{Q} \widehat{\mathbf{x}}_t \right\rVert_2 = 0 \right) = 1.
\end{equation}
\end{lemma}
To prove Lemma~\ref{lem: resilientConsensus}, we separately consider the network of agents interested in each component of $\theta^*$. We show that, for each component of $\theta^*$, all interested agents reach consensus, i.e., their estimates converge to the same value.

\begin{proof}
	Recall that the set $\mathcal{J}_m$, $m = 1, \dots, M$, is the subset of all agents interested in estimating $\left[ \theta^* \right]_m$. Let $ \mathcal{J}_m = \left\{ n_{m, 1}, \dots, n_{m, \left\lvert \mathcal{J}_m \right\rvert} \right\}.$
For each agent $n \in \mathcal{J}_m$, consider the canonical basis (row) vector (of $\mathbb{R}^{NM}$), $e^\intercal_{(n-1)M + m}$. This canonical vector selects the element of the auxiliary state $\widetilde{\mathbf{x}}_t$ corresponding to agent $n$'s estimate of $\left[\theta^* \right]_m$, i.e.,
	$e^\intercal_{(n-1)M + m} \widetilde{\mathbf{x}}_t = \left[ \widetilde{x}_n(t) \right]_m.$
Now, construct the matrix $\mathcal{Q}_m$ by stacking the canonical row vectors $e^\intercal_{(n-1)M + m}$ for all $n \in \mathcal{J}_m$, i.e.,
\begin{equation}\label{eqn: qmDef}
	\mathcal{Q}_m = \left[\begin{array}{ccc}e_{\left(n_{m, 1} - 1 \right)M + m} & \cdots & e_{\left( n_{m, \left\lvert \mathcal{J}_m \right\rvert} - 1 \right)M + m} \end{array} \right]^\intercal. 
\end{equation}
The matrix $\mathcal{Q}_m$ selects all of the interested agents' (i.e., all agents $n \in \mathcal{J}_m$) estimates of $\left[ \theta^* \right]_m$ from the auxiliary state $\widetilde{\mathbf{x}}_t$. For each component $m = 1, \dots, M$, define the $\left\lvert \mathcal{J}_m \right\rvert$-dimensional vector
	$\widehat{\mathbf{x}}_t^m = \mathcal{Q}_m \widehat{\mathbf{x}}_t,$
which collects, from all interested agents $n \in \mathcal{J}_m$, the difference between each agent's estimate of $\left[\theta^*\right]_m$ and the average estimate of all interested agents.

From the dynamics~\eqref{eqn: stackedTildeX} of the auxiliary state $\widetilde{\mathbf{x}}_t$, $\widehat{\mathbf{x}}_t^m$ evolves according to
\begin{equation}\label{eqn: hatXm}
\begin{split}
	&\widehat{\mathbf{x}}_{t+1}^m = \left(I_{\left\lvert \mathcal{J}_M \right\rvert} - P_{\left\lvert \mathcal{J}_M \right\rvert, 1} - \beta_t L_m(t)\right) \widehat{\mathbf{x}}_t^m + \\
	&\: \alpha_t \mathcal{Q}_m \left(I_{NM} - \left(\mathbf{1}_N \mathbf{1}_N^\intercal \right) \otimes \mathcal{D} \right)D^\intercal_H \mathbf{K}_t \left(\overline{\mathbf{y}}_t - D_H \widetilde{\mathbf{x}}_t \right),
\end{split}
\end{equation}
where, $L_m(t)$ is the graph Laplacian of $G_m(t)$, the subgraph induced by $\mathcal{J}_m$, and, following the definition in Lemma~\ref{lem: consensus} in the appendix,
	$P_{\left\lvert \mathcal{J}_m \right\rvert, 1} = \frac{1}{\left\lvert \mathcal{J}_m \right\rvert} \mathbf{1}_{\left\lvert \mathcal{J}_m \right\rvert} \mathbf{1}^\intercal_{\left\lvert \mathcal{J}_m \right\rvert}.$
By definition~\eqref{eqn: matrixKDef} of $\mathbf{K}_t$ and the threshold $\gamma_t$, 
we have
	$\left\lVert \mathbf{K}_t \left(\overline{\mathbf{y}}_t  - D_h \widetilde{\mathbf{x}}_t \right) \right\rVert_\infty \leq \gamma_t,$
which means that, for some finite constant, $C_1 > 0$, we have
\begin{equation}
\begin{split}
&C_1 \gamma_t \geq  \left\lVert \mathcal{Q}_m \left(I_{NM} - \left(\mathbf{1}_N \mathbf{1}_N^\intercal \right) \otimes \mathcal{D} \right)D^\intercal_H \mathbf{K}_t \left(\overline{\mathbf{y}}_t - D_H \widetilde{\mathbf{x}}_t \right) \right\rVert_2.
\end{split}
\end{equation}
As a consequence of Assumption~\ref{ass: iidGraphs}, the Laplacian matrices $\left\{ L_m(t) \right\}$ form an i.i.d. sequence, and, by Assumption~\ref{ass: connectivity}, we have $\lambda_2 \left( \mathbb{E} \left[ L_m (t) \right] \right) > 0$. The evolution~\eqref{eqn: hatXm} of $\widehat{\mathbf{x}}_t^m$ falls under the purview of Lemma~\ref{lem: consensus} in the appendix, and we have
\begin{equation}\label{eqn: resConsensus1}
	\mathbb{P} \left( \lim_{t \rightarrow \infty} \left(t+1\right)^{\tau_3} \left\lVert \widehat{\mathbf{x}}_t^m \right\rVert_2 = 0 \right) = 1,
\end{equation}
for every $0 \leq \tau_3 < \tau_\gamma + \tau_1 - \tau_2$.

To proceed, let $\overline{\mathcal{Q}}$ be the matrix $\mathcal{Q}$ with all zero \textit{rows} removed. Then, by definition, we have $\left\lVert \overline{\mathcal{Q}} \widehat{\mathbf{x}}_t \right\rVert_2 = \left\lVert \mathcal{Q} \widehat{\mathbf{x}}_t \right\rVert_2$. Note that vector $\overline{\mathcal{Q}} \widehat{\mathbf{x}}_t$ is a permutation of the vector
$\left[\begin{array}{ccc} {\widehat{\mathbf{x}}_t^{1\intercal}} & \cdots & {\widehat{\mathbf{x}}_t^{M\intercal}}\end{array} \right]^\intercal,$ which means that
	$\left\lVert \mathcal{Q} \widehat{\mathbf{x}}_t \right\rVert_2 = \left\lVert \overline{\mathcal{Q}} \widehat{\mathbf{x}}_t \right\rVert_2 \leq \sum_{m = 1}^M \left\lVert \widehat{\mathbf{x}}_t^m \right\rVert_2.$
This holds for every component $m = 1, \dots, M$, so we have
\begin{equation}\label{eqn: resConsensus3}
	\mathbb{P} \left( \lim_{t \rightarrow \infty} \left(t+1\right)^{\tau_3} \sum_{m = 1}^M \left\lVert \widehat{\mathbf{x}}_t^m \right\rVert_2 = 0 \right) = 1,
\end{equation}
for every $0 \leq \tau_3 < \tau_\gamma + \tau_1 - \tau_2$. Combining~\eqref{eqn: resConsensus3} with $\left\lVert \mathcal{Q} \widehat{\mathbf{x}}_t \right\rVert_2 \leq  \sum_{m = 1}^M \left\lVert \widehat{\mathbf{x}}_t^m \right\rVert_2$ yields~\eqref{eqn: resilientConsensus} and completes the proof.
\end{proof}

\subsection{Intermediate Result for Attack Modeling}
We present an intermediate result, which will be used to characterize the effect of the attack on $\overline{\mathbf{x}}_t$.
\begin{lemma}\label{lem: main}
 Let $A_1 \succ 0$ ($A_1 \in \mathbb{R}^{k \times k}$) be a symmetric, positive definite matrix with minimum eigenvalue $\lambda_{\min} \left(A_1 \right)$. For any $x \neq 0$, $x \in\mathbb{R}^k$, and $y \in \mathbb{R}^k$ that satisfy
$\left\lVert y \right\rVert_2 < \lambda_{\min} \left(A_1 \right) \left\lVert x \right\rVert_2,$
there exists a symmetric, positive definite matrix $A_2$ that satisfies
\begin{equation}\label{eqn: mainlem1}
	A_2 x = A_1 x + y.
\end{equation}
Moreover, the minimum eigenvalue of $A_2$ satisfies
\begin{equation}\label{eqn: mainlem2}
	\lambda_{\min} \left( A_2 \right) \geq \lambda_{\min} \left(A_1\right) - {\left\lVert {y} \right\rVert_2}{\left \lVert {x} \right\rVert^{-1}_2}.
\end{equation}
\end{lemma}
Lemma~\ref{lem: main} studies the effect of perturbations on matrix-vector multiplication $A_1 x$ when the matrix $A_1$ is postive definite. It states that, as long as the perturbation $y$ is not too strong, then, the result of the \textit{perturbed} multiplication ($A_1 x + y$), is equal to the \textit{unperturbed} multiplication $A_2 x$ between the \textit{same} vector $x$ and another postive definite matrix $A_2$. We use Lemma~\ref{lem: main} to study the effect of disturbances (from attacks) on the evolution of linear dynamical systems whose dynamics are modeled by positive definite matrices.

\begin{proof}
We separately consider the cases where $k = 1$ and $k > 1$. If $k = 1$, (i.e., $A_1$, $x$, and $y$ are all scalars), then, setting $A_2 = A_1 + \frac{y}{x}$ immediately satisfies~\eqref{eqn: mainlem1} and~\eqref{eqn: mainlem2}. For scalars $A_1$ and $A_2$, $\lambda_\min\left(A_1 \right) = A_1$ and $\lambda_{\min} \left(A_2 \right) = A_2$. Since, by definition, $\frac{\left\lvert y \right\rvert}{\left\lvert x \right\rvert} < A_1$, we have $A_2 > 0$, i.e., $A_2$ is positive definite. 

For $k > 1$,  let $\delta = {\left\lVert y \right\rVert_2}{\left\lVert x \right \rVert^{-1}_2}$, and let $\widehat{x} = \left\lVert x \right\rVert_2 e_1$, where $e_1$ is the first canonical basis vector of $\mathbb{R}^k$. For any $x \neq 0$, there exists an orthogonal matrix $V$ such that $\widehat{x} = Vx$. Now, let $\widehat{y} = V y$. 
We now show the existence of a symmetric positive definite $\widehat{A}_2$ 
that satisfies
\begin{equation}\label{eqn: mainlem3}
	\widehat{A}_2 \widehat{x} = \widehat{A}_1 \widehat{x} + \widehat{y},
\end{equation}
where $\widehat{A}_1 = V A_1 V^\intercal$. Note that, since $V$ is an orthogonal matrix, $A_1$ and $\widehat{A}_1 = V A_1 V^\intercal$ have the same eigenvalues. 

We find a symmetric matrix $\widehat{A}_3$ that satisfies $\widehat{A}_3 \widehat{x} = \widehat{y}$. Partition the vector $\widehat{y}$ as
	$\widehat{y} = \left[ \begin{array}{c} \widehat{y}_1 \\ \widehat{y}_2 \end{array} \right],$
where $\widehat{y}_1 \in \mathbb{R}$ is the first element of $\widehat{y}$ and $\widehat{y}_2 \in \mathbb{R}^{k-1}$ is the vector of the remaining elements. Define the matrix $\widehat{A}_3$ as
\begin{equation}\label{eqn: mainlem5}
	\widehat{A}_3 = \left[\begin{array}{cc} \widehat{y}_1 \left\lVert x \right\rVert_2^{-1} &  \widehat{y}_2^\intercal \left\lVert x \right\rVert_2^{-1} \\  \widehat{y}_2 \left\lVert x \right\rVert_2^{-1} & \nu I_{k-1} \end{array} \right],
\end{equation}
where $\nu > 0$ and $I_{k-1}$ is the $k-1$ by $k-1$ identity matrix. Note that $\widehat{A}_3$ is a symmetric matrix that satisfies $\widehat{A}_3 \widehat{x} = \widehat{y}$. Setting $\widehat{A}_2 = \widehat{A}_1 + \widehat{A}_3$ satisfies~\eqref{eqn: mainlem3}. What remains is to show that $\widehat{A}_1 + \widehat{A}_3$ is positive definite and satisfies $\lambda_{\min}\left({\widehat{A}_1 + \widehat{A}_3 }\right) \geq \lambda_{\min} \left(\widehat{A}_1 \right) - \delta$. 

We express $\widehat{A}_2$ as
	$\widehat{A}_2 = \left(\widehat{A}_1 - \delta I_k \right) + \left(\widehat{A}_3 + \delta I_k \right).$
Since $\widehat{A}_1$ is a symmetric, positive definite matrix, it is diagonalizable over the reals: there exists an orthogonal matrix $\widehat{U}_1$ and a diagonal matrix $\Lambda_1$ such that
	$\widehat{A}_1 = \widehat{U}_1 \Lambda_1 \widehat{U}_1^\intercal,$
which means that
	$\widehat{A}_1 - \delta I_k = \widehat{U}_1 \left( \Lambda_1 - \delta I_k \right) \widehat{U}_1^\intercal,$
and
\begin{equation}\label{eqn: mainlem8}
	\lambda_{\min} \left( \widehat{A}_1 - \delta I_k \right) = \lambda_{\min} \left( \widehat{A}_1 \right) - \delta. 
\end{equation}

To proceed, we find $\nu > 0$ (from the definition of $\widehat{A}_3$~\eqref{eqn: mainlem5}) such that $\widehat{A}_3 + \delta I_k$ is positive semidefinite. From~\eqref{eqn: mainlem5}, we have
\begin{equation}\label{eqn: mainlem9}
	\widehat{A}_3 + \delta I_k =  \left[\begin{array}{cc} \widehat{y}_1 \left\lVert x \right\rVert_2^{-1} + \delta &  \widehat{y}_2^\intercal \left\lVert x \right\rVert_2^{-1} \\  \widehat{y}_2 \left\lVert x \right\rVert_2^{-1} & (\nu + \delta) I_{k-1} \end{array} \right].
\end{equation}
Applying Proposition 16.1 from~\cite{Schur} (Lemma~\ref{lem: schur} in the appendix), we have that $\widehat{A}_3 + \delta I_k \succeq 0$ if and only if
\begin{equation}\label{eqn: mainlem10}
	\widehat{y}_1 \left\lVert x \right\rVert_2^{-1} + \delta - \left\lVert \widehat{y}_2 \right\rVert_2^2 \left\lVert x \right\rVert_2^{-2} \left( \nu + \delta \right)^{-1} \geq 0.
\end{equation}
Let $b = \widehat{y}_1{\left\lVert x \right\rVert_2^{-1}}$. Since $\left \lVert \widehat{y} \right\rVert_2 = \left\lVert V y \right\rVert_2 = \left \lVert y \right\rVert_2$, 
we have
$\left\lvert b \right \rvert = {\left \lvert \widehat{y}_1 \right\rvert}{ \left\lVert x \right\rVert^{-1}_2} \leq {\left\lVert \widehat{y} \right \rVert_2}{\left\lVert x \right\rVert^{-1}_2} \! = \! \delta.$
We express $\left\lVert \widehat{y}_2 \right\rVert^2_2\left\lVert x \right\rVert^{-2}_2$ as
\begin{equation} \label{eqn: mainlem12}
	 {\left\lVert \widehat{y}_2 \right\rVert^{-2}_2}{\left\lVert x \right\rVert^2_2} = {\left(\left\lVert \widehat{y} \right\rVert^2_2 - \widehat{y}_1^2 \right)} { \left\lVert x \right\rVert_2^{-2}} = \delta^2 - b^2.
\end{equation}
Using~\eqref{eqn: mainlem12}, and performing algebraic manipulations,~\eqref{eqn: mainlem10} becomes
\begin{equation}\label{eqn: mainlem13}
	b^2 + \left(\delta + \nu\right) b + \delta \nu \geq 0.
\end{equation}
The roots of the expression on the left hand side of~\eqref{eqn: mainlem13} are $b = -\nu$ and $b = -\delta$. That is, as long as $\nu \geq \delta$, then, $b^2 + \left(\delta + \nu\right) b + \delta \nu \geq 0$ for all $-\delta \leq b \leq \delta$, which means that $\widehat{A}_3 + \delta I_k \succeq 0$ as long as $\nu \geq \delta$.  

The sum of a positive definite and a positive semidefinite matrix is positive definite. Since $\widehat{A}_3 + \delta I_k \succeq 0$ and $\widehat{A}_1 - \delta I_k \succ 0$, we have 
\begin{equation}\label{eqn: mainlem14}
	\widehat{A}_2 = \left(\widehat{A}_1 - \delta I_k \right) + \left(\widehat{A}_3 + \delta I_k \right) \succ 0. 
\end{equation}
Moreover, since  $\widehat{A}_3 + \delta I_k \succeq 0$, we have
\begin{equation}\label{eqn: mainlem15}
	\lambda_{\min} \left( \widehat{A}_2 \right) \geq \lambda_{\min} \left( \widehat{A}_ 1 - \delta I_k \right) =\lambda_{\min} \left( \widehat{A}_1 \right) - \delta. 
\end{equation}
Setting $A_2 = V^\intercal \widehat{A}_2 V$ yields~\eqref{eqn: mainlem1} and~\eqref{eqn: mainlem2}, which completes the proof.
\end{proof}

\subsection{Generalized Network Average Behavior}
This subsection analyzes the evolution of the generalized network average estimate $\overline{\mathbf{x}}_t$. In particular, we show that $\overline{\mathbf{x}}_t$ converges a.s. to $\theta^*$ under the resilience condition~\eqref{eqn: resilienceCondition}. Let
\begin{equation}\label{eqn: avgError}
	\overline{\mathbf{e}}_t = \overline{\mathbf{x}}_t - \theta^*
\end{equation}
be the generalized network average estimate error. To analyze the behavior of $\overline{\mathbf{e}}_t$, we require the following definitions:
\begin{align}
	\color{black} \widetilde{k}_p(t) & \color{black} = \left\{ \begin{array}{ll} k_p(t), & p \in \mathcal{N}, \\ 0, & p \in \mathcal{A}, \end{array} \right.\\
	\mathbf{K}^{\mathcal{N}}_t &= \diag \left(\widetilde{k}_1(t), \dots, \widetilde{k}_p(t) \right), \\
	\mathbf{K}^{\mathcal{A}}_t &= \mathbf{K}_t - \mathbf{K}^{\mathcal{N}}_t, \label{eqn: ka} \\
	\overline{\mathbf{w}}_t &= \left[ \begin{array}{ccc} \overline{w}_1(t)^\intercal & \cdots & \overline{w}_N(t)^\intercal \end{array} \right]^\intercal. 
\end{align}

The following lemma characterizes the relationship between the generalized network average estimate and the threshold $\gamma_t$. In our previous work~\cite{ChenSAGE}, we established a similar result for resilient distributed \textit{parameter} estimation (Lemma 2 in~\cite{ChenSAGE}). Here, we consider the resilient distributed field estimation problem.
\begin{lemma}\label{lem: avg1}
	Define the auxiliary threshold
\begin{equation}\label{eqn: lineGammaDef}
	\color{black} \overline{\gamma}_t = \gamma_t - \frac{X}{(t+1)^{\tau_3}} - \frac{W}{(t+1)^{\frac{1}{2} - \epsilon_W}},
\end{equation}
where{\color{black}, recall, from the weight selection procedure in Section~\ref{sect: algorithmDescription}, $\gamma_t = \frac{\Gamma}{(t+1)^{\tau_{\gamma}}}$ ($\Gamma >0$ is an arbitrary positive constant)}, $\tau_3 = \tau_\gamma + \tau_1 - \tau_2 - \epsilon_X$ for arbitrarily small $0 < \epsilon_X < \tau_1 - \tau_2, \: 0 < \epsilon_W < \frac{1}{2} - \tau_\gamma.$
As long as $\lambda_{\min} \left( \mathcal{G}_{\mathcal{N}} \right) > \Delta_{\mathcal{A}}$ (resilience condition~\eqref{eqn: resilienceCondition}), then there exists a.s. $T_0 \geq 0$, $0 < X < \infty$, and $0 < W < \infty$ such that
\begin{enumerate}
	\item $\left\lVert \mathcal{Q}\widehat{\mathbf{x}}_t \right\rVert_2 \leq \frac{X}{(t+1)^{\tau_3}}$ a.s.,
	\item $\left\lVert \overline{\mathbf{w}}_t \right\rVert_2 \leq \frac{W}{(t+1)^{\frac{1}{2} - \tau_\gamma}}$, a.s., and,
	\item if, for any $T \geq T_0$, $\left\lVert \overline{\mathbf{e}}_T \right \rVert_2 \leq \overline{\gamma}_T$, then, for all $t \geq T$, $\left\lVert \overline{\mathbf{e}}_t \right\rVert_2 \leq \overline{\gamma}_t$ a.s.
\end{enumerate}
\end{lemma}
\noindent Compared to the proof of Lemma 2 in~\cite{ChenSAGE} (for distributed parameter estimation), the proof of Lemma~\ref{lem: avg1} for the distributed field estimation setup requires new technical tools that we develop in this paper. Specifically, the proof of Lemma~\ref{lem: avg1} requires Lemma~\ref{lem: main} from the previous subsection to analyze the effect of attacks. 

To prove Lemma~\ref{lem: avg1}, we show that conditions 1) and 2) hold as a result of existing Lemmas (Lemma~\ref{lem: resilientConsensus} and Lemma~\ref{lem: measureNoise} in the appendix). We then study the evolution of the network average estimation error $\overline{\mathbf{e}}_t$ along sample paths that satisfy conditions 1) and 2). Along these sample paths, we show that condition 3) holds by analyzing the estimate update step of \textbf{SAFE}. Finally, since the set of sample paths on which conditions 1) and 2) hold has probability measure $1$, we show that condition 3) also holds with probability $1$. 

\begin{proof}
	\underline{Step 1 (Satisfying conditions 1) and 2).)}: By Lemmas~\ref{lem: resilientConsensus} and~\ref{lem: measureNoise} (in the appendix), we have
\begin{align}\label{eqn: staysBounded1}
	\mathbb{P} \left(\lim_{t \rightarrow \infty} (t+1) ^{\tau_3} \left \lVert \mathcal{Q} \widehat{\mathbf{x}}_t \right \rVert_2 = 0 \right) &= 1, \\
\mathbb{P} \left(\lim_{t \rightarrow \infty} (t+1)^{\delta_0} \left\lVert \overline{\mathbf{w}}_t \right\rVert_2 = 0 \right) &= 1, \label{eqn: staysBounded2}
\end{align}
for every $0 \leq \tau_3 < \tau_\gamma + \tau_1 - \tau_2$ and every $0 \leq \delta_0 < \frac{1}{2}.$ That is, the set of sample paths $\Omega ^' \subset \Omega$ on which
	$\lim_{t \rightarrow \infty} (t+1)^{\delta_0} \left\lVert \overline{\mathbf{w}}_{t, \omega}\right\rVert_2 = 0$ and 
	$\lim_{t \rightarrow \infty} (t+1) ^{\tau_3} \left \lVert \widehat{\mathbf{x}}_{t, \omega} \right \rVert_2 = 0$ 
has probability measure $1$. For each sample path $\omega \in \Omega^'$, there exists finite $X_\omega$ and $W_\omega$ such that
\begin{align}
	 \left\lVert \mathcal{Q} \widehat{\mathbf{x}}_{t, \omega} \right \rVert_2 \leq \frac{X_\omega} {(t+1)^{\tau_3}}, \: \left\lVert \overline{\mathbf{w}}_{t, \omega} \right \rVert_2 &\leq \frac{W_\omega} {(t+1)^{\frac{1}{2} - \epsilon_W}},
\end{align}
where the notation $\widehat{\mathbf{x}}_{t, \omega}$ and $\overline{\mathbf{w}}_{t, \omega}$ mean $\widehat{\mathbf{x}}_t$ and $\overline{\mathbf{w}}_t$ along the sample path $\omega$.

\underline{Step 2 (Dynamics of $\overline{\mathbf{e}}_t$)}:We now consider the evolution of $\overline{\mathbf{e}}_t$. From~\eqref{eqn: stackedTildeX}, we have that $\overline{\mathbf{x}}_t$ evolves according to
\begin{equation}\label{eqn: aux1}
	\overline{\mathbf{x}}_{t+1} = \overline{\mathbf{x}}_t + \alpha_t \mathcal{D} \left( \mathbf{1}^\intercal_N \otimes I_M \right)D_H^\intercal \mathbf{K}_t \left(\overline{\mathbf{y}}_t - D_H \widetilde{\mathbf{x}}_t \right).
\end{equation}
To derive~\eqref{eqn: aux1} from~\eqref{eqn: stackedTildeX}, we have used the fact that $\left(\mathbf{1}_N^\intercal \otimes I_M \right)  \mathbf{L}_t = 0$. Recall, the definition of $Q_n$, that for each agent $n$, $H_n Q_n = H_n$, which means that 
	$D_H \mathcal{Q} = D_H.$
We express $D_H \mathbf{x}_t$ as
\begin{align}\label{eqn: aux3}
	D_H \mathbf{x}_t &= D_H \left( \mathbf{1}_N \otimes I_M \right) \overline{\mathbf{x}}_t + D_H \mathcal{Q} \widehat{\mathbf{x}}_t.
\end{align}
Then, using the fact that $\mathbf{K}_t = \mathbf{K}_t^{\mathcal{N}} + \mathbf{K}_t^{\mathcal{A}}$ and substituting~\eqref{eqn: aux3} into~\eqref{eqn: aux1}, we find that $\overline{\mathbf{e}}_t = \overline{\mathbf{x}}_t - \theta^*$ evolves according to
\begin{equation}\label{eqn: averageErrorDynamics}
	\begin{split}
		\overline{\mathbf{e}}_{t+1} &= \left( I_{M} - \alpha_t\mathcal{D} \sum_{p \in \mathcal{N}} k_p(t) h_p h_p^\intercal \right) \overline{\mathbf{e}}_t - \\
		& \quad \alpha_t \mathcal{D} \left( \mathbf{1}_N^\intercal \otimes I_M \right) D_H^\intercal \left( \mathbf{K}^{\mathcal{N}}_t \left(D_H\mathcal{Q}\widehat{\mathbf{x}}_t - \overline{\mathbf{w}}_t \right) +  \mathbf{K}^{\mathcal{A}}_t \left( \overline{\mathbf{y}}_t - D_H \widetilde{\mathbf{x}}_t \right)\right)
	\end{split}
\end{equation}

Let $\mathbf{b}_t = \left(\mathbf{1}_N^\intercal \otimes I_M \right) D^\intercal_H \mathbf{K}_t^{\mathcal{A}} \left( \overline{\mathbf{y}}_t - D_H \widetilde{\mathbf{x}}_t \right)$ model the effect of the attack on the evolution of $\overline{\mathbf{e}}_t$. Recall that, by definition of the $k_p(t)$ and $\mathbf{K}_t^{\mathcal{A}}$ (equations~\eqref{eqn: smallK} and~\eqref{eqn: ka}), we have $\left\lVert \mathbf{K}_t^{\mathcal{A}} \left( \overline{\mathbf{y}}_t - D_H \widetilde{\mathbf{x}}_t \right) \right\rVert_\infty \leq \gamma_t$. Then, from the definition of $\Delta_{\mathcal{A}}$, we have
\begin{equation}\label{eqn: aux4}
	\left\lVert \mathbf{b}_t \right\rVert_2 \leq \Delta_{\mathcal{A}} \gamma_t. 
\end{equation}
Note that $\gamma_t = \overline{\gamma}_t + X(t+1)^{\tau_3} + W(t+1)^{\frac{1}{2} - \epsilon_W}$, which means that we may partition $\mathbf{b}_t$ as 
\begin{equation}\label{eqn: aux4a}
	\mathbf{b}_t = \overline{\mathbf{b}}_t + \widetilde{\mathbf{b}}_t,
\end{equation}
where $\left\lVert \widetilde{\mathbf{b}}_t \right\rVert_2 \leq  \Delta_{\mathcal{A}} \left(X(t+1)^{\tau_3} + W(t+1)^{\frac{1}{2} - \epsilon_W}\right)$ and $\left\lVert \overline{\mathbf{b}}_t \right\rVert_2 \leq \Delta_{\mathcal{A}} \overline{\gamma}_t$.

\underline{Step 3 (Pathwise analysis of $\overline{\mathbf{e}}_t$)}: Next, we study the evolution of $\overline{\mathbf{e}}_{t, \omega}$ for sample paths $\omega \in \Omega ^'$. 
We show that, for sufficiently large $T_0$, if, for some $T \geq T_0$, $\left\lVert \overline{\mathbf{e}}_{T, \omega} \right\rVert_2 \leq \overline{\gamma}_{T, \omega}$, then, for all $ t \geq T$, $\left\lVert \overline{\mathbf{e}}_{t, \omega} \right\rVert_2 \leq \overline{\gamma}_{t, \omega}$. For the noncompromised measurement streams $p \in \mathcal{N}$, by the triangle inequality, we have
\begin{align}
	\left\lvert\overline{y}^{(p)}_n(T, \omega) - h_p^\intercal x_n(T, \omega) \right \rvert &\leq \left\lvert h_p^\intercal \left(Q_n \left(\overline{\mathbf{x}}_{T, \omega} - \widetilde{x}_n\left(T, \omega \right) \right) - \overline{\mathbf{e}}_t \right) \right\rvert + \left\lvert \overline{w}^{p}_n(T, \omega) \right \rvert, \\
	&\leq \left\lVert \overline{\mathbf{e}}_{T, \omega} \right \rVert_2 + \left\lVert \mathcal{Q}\widehat{\mathbf{x}}_{T, \omega} \right\rVert_2 + \left\lVert \overline{\mathbf{w}}_{T, \omega} \right \rVert_2,\label{eqn: triangle1}\\
&\leq \overline{\gamma}_{T, \omega} + \frac{X_\omega}{(T+1)^{\tau_3}} + \frac{W_\omega}{(T+1)^{\frac{1}{2} - \epsilon_W}} = \gamma_t \label{eqn: aux5},
\end{align}
which means that $k_p(t) = 1$ for all $p \in \mathcal{N}$. Then, using the fact that $\sum_{p \in \mathcal{N}} h_p h_p^\intercal = \mathcal{G}_{\mathcal{N}}$, from~\eqref{eqn: averageErrorDynamics}, we have
\begin{equation}\label{eqn: aux6}
	\begin{split}
		&\overline{\mathbf{e}}_{T+1, \omega} = \mathbf{e}_{T, \omega} - \alpha_T\mathcal{D} \left(\mathcal{G}_{\mathcal{N}} \overline{\mathbf{e}}_{T, \omega} + \overline{\mathbf{b}}_{T, \omega} + \widetilde{\mathbf{b}}_{T, \omega} \right) - \\
		&\quad \alpha_T \mathcal{D} \left( \mathbf{1}_N^\intercal \otimes I_M \right) D_H^\intercal \mathbf{K}^{\mathcal{N}}_{T, \omega} \left(D_H\mathcal{Q}\widehat{\mathbf{x}}_{T, \omega} - \overline{\mathbf{w}}_{T, \omega} \right).\\
	\end{split}
\end{equation}
Further, using the fact that $\Delta_{\mathcal{A}} \leq \left\lvert \mathcal{A} \right\rvert$ and $\left\lvert \mathcal{A} \right\rvert + \left\lvert \mathcal{N} \right\rvert = P$, we have

\begin{equation}\label{eqn: aux6a}
	\begin{split}
		\! \left\lVert \overline{\mathbf{e}}_{T+1, \omega} \right\rVert_2 \! \leq \! \left\lVert  \left( I_{M} - \alpha_T\mathcal{D} \mathcal{G}_{\mathcal{N}} \right) \overline{\mathbf{e}}_{T, \omega} \!\!+\!\! \alpha_T \mathcal{D} \overline{\mathbf{b}}_{T, \omega} \right\rVert_2 \! + \! \frac{\alpha_T P X_\omega}{\underline{J}(T+1)^{\tau_3}} \!\!+\!\! \frac{\alpha_T P W_\omega}{\underline{J}(T+1)^{\frac{1}{2} - \epsilon_W}},\!\!
	\end{split}
\end{equation}
where $\underline{J} = \min_{m = 1, \dots, M} \left\lvert \mathcal{J}_m \right\rvert.$ 

Let
\begin{equation}\label{eqn: aux7}
\begin{split}
	&\overline{\mathbf{e}}_{T, \omega}^*= \argmax_{\overline{\mathbf{e}} \in \mathbb{R}^M, \: \left\lVert \overline{\mathbf{e}} \right\rVert_2 \leq \overline{\gamma}_{T, \omega}} \left\lVert  \Big( I_{M} - \alpha_T\mathcal{D} \mathcal{G}_{\mathcal{N}} \Big) \overline{\mathbf{e}} + \alpha_T \mathcal{D} \overline{\mathbf{b}}_{T, \omega} \right\rVert_2.
\end{split}
\end{equation}
Since $\left\lVert \overline{\mathbf{e}}_{T, \omega} \right\rVert_2 \leq \overline{\gamma}_{T, \omega}$, $\overline{\mathbf{e}}_{T, \omega}$ is a feasible solution for the optimization problem above, which means that
\begin{equation}\label{eqn: aux8}
	\begin{split}
		\! \left\lVert \overline{\mathbf{e}}_{T+1, \omega} \right\rVert_2 \! \leq \! \left\lVert  \left( I_{M} - \alpha_T\mathcal{D} \mathcal{G}_{\mathcal{N}} \right) \overline{\mathbf{e}}^*_{T, \omega} \!\!+\!\! \alpha_T \mathcal{D} \overline{\mathbf{b}}_{T, \omega} \right\rVert_2 \! + \! \frac{\alpha_T P X_\omega}{\underline{J}(T+1)^{\tau_3}} \!\!+\!\! \frac{\alpha_T P W_\omega}{\underline{J}(T+1)^{\frac{1}{2} - \epsilon_W}}.\!\!
	\end{split}
\end{equation}
Because $\overline{\mathbf{e}}_{T, \omega}^*$ maximizes a convex function over a closed convex set, it belongs to the boundary of the convex set, i.e., $\left\lVert \overline{\mathbf{e}}_{T, \omega}^* \right\rVert_2 = \overline{\gamma}_{T, \omega}.$ Note that $\left\lVert \overline{b}_{T, \omega} \right\rVert_2 \leq \overline{\gamma}_{T, \omega} < \lambda_\min \left( \mathcal{G}_N \right) \left\lVert \overline{\mathbf{e}}_{T, \omega}^* \right \rVert_2$, which, by Lemma~\ref{lem: main}, means that there exists $\mathcal{G}^*_{T, \omega} \succ 0$ such that 
\begin{equation}\label{eqn: aux9}
	\mathcal{G}^*_{T, \omega} \overline{\mathbf{e}}_{T, \omega}^* = \mathcal{G}_{\mathcal{N}} \overline{\mathbf{e}}_{T, \omega}^* + \overline{\mathbf{b}}_{T, \omega}
\end{equation} 
with a minimum eigenvalue that satisfies
\begin{equation}\label{eqn: kappaDef}
	\lambda_\min \left( \mathcal{G}^*_{T, \omega} \right) \geq \kappa = \lambda_{\min} \left( \mathcal{G}_{\mathcal{N}} \right) - \Delta_{\mathcal{A}}.
\end{equation}
Substituting~\eqref{eqn: aux9} into~\eqref{eqn: aux8}, we have
\begin{equation}\label{eqn: aux10}
	\begin{split}
		\left\lVert \overline{\mathbf{e}}_{T+1, \omega} \right\rVert_2 & \leq \left\lVert  \left( I_{M} - \alpha_T\mathcal{D} \mathcal{G}^*_{T, \omega} \right) \overline{\mathbf{e}}^*_{T, \omega}\right\rVert_2 + \frac{\alpha_T P X_\omega}{\underline{J}(T+1)^{\tau_3}} \!\!+\!\! \frac{\alpha_T P W_\omega}{\underline{J}(T+1)^{\frac{1}{2} - \epsilon_W}}.
	\end{split}
\end{equation}
The matrix $\mathcal{D}\mathcal{G}_{T, \omega}^*$ is similar to the matrix $\mathcal{D}^{\frac{1}{2}}\mathcal{G}_{T, \omega}^*\mathcal{D}^{\frac{1}{2}}.$ Using the fact that, for $T$ large enough, $\left\lVert I_{M} - \alpha_T\mathcal{D} \mathcal{G}^*_{T, \omega}\right\rVert_2 \leq 1 - \frac{\alpha_T \kappa}{\overline{J}},$ where $\overline{J} = \max_{m = 1, \dots, M} \left\lvert \mathcal{J}_M \right\rvert$,~\eqref{eqn: aux10} becomes
\begin{equation}\label{eqn: aux11}
	\left\lVert \overline{\mathbf{e}}_{T+1, \omega} \right\rVert_2  \leq \left(1 - \frac{\alpha_T \kappa}{\overline{J}} \right)\overline{\gamma}_{T, \omega} + \frac{\alpha_T P \left(X_\omega + W_\omega \right)}{\underline{J}(T+1)^{\tau_4}},
\end{equation}
where $\tau_4 = \min\left(\tau_3, \frac{1}{2} - \epsilon_W \right).$

\underline{Step 4 (Satisfying condition 3.)}: From~\eqref{eqn: aux11}, we show that for $T$ large enough, $\left\lVert \overline{\mathbf{e}}_{T+1} \right\rVert_2 \leq \overline{\gamma}_{T+1, \omega}$. It suffices to show that 
\begin{equation}\label{eqn: aux12}
	\left\lVert \overline{\mathbf{e}}_{T+1} \right\rVert_2 \leq \left(\frac{T+1}{T+2} \right)^{\tau_\gamma}\overline{\gamma}_{T, \omega},
\end{equation} since $\left(\frac{T+1}{T+2} \right)^{\tau_\gamma}\overline{\gamma}_{T, \omega} < \overline{\gamma}_{T+1, \omega}.$ Note that we may express $\overline{\gamma}_t$ as
\begin{equation}\label{eqn: aux13}
	\overline{\gamma}_t = \frac{\Gamma - \frac{X}{(t+1)^{\tau_3 - \tau_\gamma}}  - \frac{W}{(t+1)^{\frac{1}{2} - \epsilon_W - \tau_\gamma}}}{(t+1)^{\tau_\gamma}}.
\end{equation}
The numerator in~\eqref{eqn: aux13} increases in $t$, which means that, for any $0 < \overline{\Gamma} < \Gamma$, there exists a finite $T$ such that $\overline{\gamma}_t > \frac{\overline{\Gamma}}{(t+1)^{\tau_{\gamma}}}$ for all $t \geq T$. Thus, for any $0 < \overline{\Gamma} < \Gamma$, there exists a finite $T$ that is sufficiently large so that~\eqref{eqn: aux11} becomes
\begin{equation}\label{eqn: aux14}
	\left\lVert \overline{\mathbf{e}}_{T+1, \omega} \right\rVert_2  \leq \left(1 - \alpha_{T} \rho_{T, \omega} \right)\overline{\gamma}_{T, \omega},
\end{equation}
where \begin{equation}\label{eqn: rhoDef} \rho_{T, \omega} = \frac{\kappa}{\overline{J}} - \frac{P(X_\omega + W_\omega)}{\overline{\Gamma}\underline{J} (T+1)^{\tau_4 - \tau_\gamma}}. \end{equation} The second term in $\rho_{T, \omega}$ decays to $0$ as $T$ increases, so, for sufficiently large $T$, $\rho_{T, \omega} \geq 0$. 

Substituting for~\eqref{eqn: aux14} and applying the inequality $(1-x) \leq e^{-x}$ for  $x \geq 0$,~\eqref{eqn: aux12} becomes $-\alpha_T \rho_{T, \omega} \leq \tau_\gamma \log \left( \frac{T+1}{T+2} \right)$. Further applying the inequality $\log \left(\frac{T + 1} {T + 2} \right) \geq 1 -\frac{T + 2} {T + 1} = -\frac{1}{T + 1}$, a sufficient condition for~\eqref{eqn: aux12} is
\begin{equation}\label{eqn: aux15}
	\alpha_T \rho_{T, \omega} \geq \frac{\tau_\gamma}{T+1}.
\end{equation}
By definition of $\rho_{T, \omega}$, for any $0 < \overline{\rho} < \frac{\kappa}{\overline{J}}$, there exists a sufficiently large finite $T$, such that $\rho_{T, \omega} > \overline{\rho}$. Then, for sufficiently large $T$, a sufficient condition for~\eqref{eqn: aux15} becomes
	$\frac{a \overline{\rho}}{(T+1)^{\tau_1}} >\frac{ \tau_\gamma}{T+1},$
which is satisfied for all $T \geq \left(\frac{\tau_\gamma}{a \overline{\rho} }\right)^{\frac{1}{{1-\tau_1}}}- 1.$ Thus, there exists a sufficiently large finite $T_0$ such that, if $\left\lVert \overline{\mathbf{e}}_{T, \omega} \right\rVert_2 \leq \overline{\gamma}_{T, \omega}$ for any $T \geq T_0$, then $\left\lVert \overline{\mathbf{e}}_{T + 1, \omega} \right\rVert_2 \leq \overline{\gamma}_{T + 1, \omega}$. The same analysis above applies to all $t = T+1, T+2, \dots,$ which means that $\left\lVert \overline{\mathbf{e}}_{t, \omega} \right\rVert_2 \leq \overline{\gamma}_{t, \omega}$ for all $t \geq T$. This holds on all sample paths $\omega \in \Omega^'$, i.e., almost surely. 
\end{proof}

Using the previous result (Lemma~\ref{lem: avg1}), we now study the convergence of the generalized network average estimate to the true value of the parameter. 
\begin{lemma}[Generalized Network Average Consistencys]\label{lem: avg2}
	  If $\lambda_{\min}\left(\mathcal{G}_{\mathcal{N}} \right) > \Delta_{\mathcal{A}},$ then, under \textbf{SAFE}, for every $0 \leq \tau_0 < \min\left(\tau_\gamma, \frac{1}{2} - \tau_\gamma \right)$,
\begin{equation}\label{eqn: averageConv}
	\mathbb{P} \left( \lim_{t \rightarrow \infty} (t+1)^{\tau_0} \left\lVert \overline{\mathbf{e}}_t \right \rVert_2 = 0 \right) = 1.
\end{equation}
\end{lemma}
We prove Lemma~\ref{lem: avg2} by studying the behavior of $\overline{\mathbf{e}}_{t}$ pathwise over the set of sample paths where Lemma~\ref{lem: avg1} holds. Along each sample path, we show that \textbf{SAFE} bounds the perturbance from the attack on the evolution of $\overline{\mathbf{e}}_t$. We use Lemma~\ref{lem: main} to model the \textit{perturbed} (as a result of the attack) dynamics of $\overline{\mathbf{e}}_t$ using an \textit{unperturbed} dynamics matrix, and then use Lemma~\ref{lem: timeVaryingSystem1} in the appendix to show that $\left\lVert \overline{\mathbf{e}}_t \right\rVert_2$ converges to $0$.
\begin{proof}
 Consider the set of sample paths $\Omega^'$ on which Lemma~\ref{lem: avg1} holds. The set $\Omega^'$ has probability measure $1$. For each sample path $\omega \in \Omega^'$, there exists finite $T_{0, \omega}$ such that, if at any $T \geq T_{0, \omega}$, we have $\left\lVert \overline{\mathbf{e}}_{T, \omega} \right \rVert_2 \leq \overline{\gamma}_{T, \omega}$, then, we have $\left \lVert \overline{\mathbf{e}}_{t, \omega} \right\rVert_2 \leq \overline{\gamma}_{t, \omega}$ for all $t \geq T$. We now analyze $\left\lVert \overline{\mathbf{e}}_{t, \omega} \right\rVert_2$ for $t \geq T_{0, \omega}$. For each sample path $\omega$, there are two possibilities: either 1. there exists $T$ such that $\left\lVert \overline{\mathbf{e}}_{T, \omega} \right \rVert_2 \leq \overline{\gamma}_{T, \omega}$, or 2. $\left\lVert \overline{\mathbf{e}}_{t, \omega} \right \rVert_2 \geq \overline{\gamma}_{t, \omega}$ for all $t \geq T_{0, \omega}.$ If the first case occurs, then, as a consequence of Lemma~\ref{lem: avg1}, we have
	$\left\lVert \overline{\mathbf{e}}_{t, \omega} \right \rVert_2 \leq \frac{\Gamma}{(t+1)^{\tau_\gamma}},$
for all $t \geq T_{0, \omega}$, which means that $\lim_{t \rightarrow \infty} (t+1)^{\tau_0} \left\lVert \overline{\mathbf{e}}_{t, \omega} \right\rVert_2 = 0$ for every $0 \leq \tau_0 < \tau_\gamma$. 

If the second case occurs, then, for all $t \geq T_{0, \omega}$, we have $\left\lVert \overline{\mathbf{e}}_{t, \omega} \right \rVert_2 > \overline{\gamma}_{t, \omega}$. Define
\begin{equation}\label{eqn: avgConv2}
	\widehat{K}_{t, \omega} = \frac{\overline{\gamma}_{t, \omega} + \frac{X_\omega}{(t+1)^{\tau_3}} + \frac{W_\omega}{(t+1)^{\frac{1}{2} - \epsilon_W}}}{\left\lVert \overline{\mathbf{e}}_{t, \omega} \right \rVert_2  + \frac{X_\omega}{(t+1)^{\tau_3}} + \frac{W_\omega}{(t+1)^{\frac{1}{2} - \epsilon_W}}}. 
\end{equation}
The numerator in~\eqref{eqn: avgConv2}  is equal to the threshold $\gamma_t$. Since the sample path $\omega$ is chosen from a set on which Lemma~\ref{lem: avg1} holds, we have $\left\lVert \mathcal{Q} \widehat{\mathbf{x}}_{t, \omega} \right \rVert_2 \leq \frac{X_\omega}{(t+1)^{\tau_3}}$ and $\left\lVert \overline{\mathbf{w}}_t \right \rVert_2 \leq \frac{W_\omega}{(t+1)^{\frac{1}{2} - \epsilon_W}}$. Thus, as a consequence of~\eqref{eqn: triangle1}, the denominator in~\eqref{eqn: avgConv2} is greater than or equal to $\left\lvert\overline{y}^{(p)}_n(t, \omega) - h_p^\intercal x_n(t, \omega) \right\rvert$ for all uncompromised measurement streams $p \in \mathcal{N}$, which means that, for all $p \in \mathcal{N}$, we have
\begin{equation}\label{eqn: avgConv3}
	\widehat{K}_{t, \omega} \leq k_p \left(t, \omega \right).
\end{equation}
That is, $\widehat{K}_{t, \omega}$ is a lower bound on the saturating gains $k_p(t, \omega)$ for all measurement uncompromised measurement streams. Define the matrix 
	$\widetilde{\mathcal{G}}_{\mathcal{N}}^{t, \omega} = \sum_{p \in \mathcal{N}} k_p(t, \omega) h_p h_p^\intercal,$
and note that, as a consequence of~\eqref{eqn: avgConv3}, we have $\widetilde{\mathcal{G}}_{\mathcal{N}}^{t, \omega} \succ \widehat{K}_{t, \omega} \mathcal{G}_{\mathcal{N}}$, where, recall, $\mathcal{G}_{\mathcal{N}} = \sum_{p \in \mathcal{N}} h_p h_p^\intercal.$

Rearranging~\eqref{eqn: avgConv2}, we have
\begin{equation}\label{eqn: avgConv4}
	\gamma_t = \widehat{K}_{t, \omega} \left(\left\lVert \overline{\mathbf{e}}_{t, \omega} \right \rVert_2  + \frac{X_\omega}{(t+1)^{\tau_3}} + \frac{W_\omega}{(t+1)^{\frac{1}{2} - \epsilon_W}} \right).
\end{equation}
Recall, from the proof of Lemma~\ref{lem: avg1}, that we let $\mathbf{b}_t = \Big(\mathbf{1}_N^\intercal \otimes I_M \Big) D^\intercal_H \mathbf{K}_t^{\mathcal{A}} \Big( \overline{\mathbf{y}}_t - D_H \widetilde{\mathbf{x}}_t \Big)$ represent the effect of the attack on the evolution of $\overline{\mathbf{e}}_t$, and, by~\eqref{eqn: aux4}, we have $\left\lVert  \mathbf{b}_t \right\rVert \leq \Delta_{\mathcal{A}} \gamma_t$. For each sample path $\omega$, we partition $\mathbf{b}_{t, \omega}$ (differently from the partition in~\eqref{eqn: aux4a}) as $\mathbf{b}_{t, \omega} = \overline{\mathbf{b}}_{t, \omega} + \widetilde{\mathbf{b}}_{t, \omega},$ where 
\begin{align}
	\left\lVert \overline{\mathbf{b}}_{t, \omega} \right\rVert_2 &\leq \widehat{K}_{t, \omega} \Delta_{\mathcal{A}} \left\lVert \overline{\mathbf{e}}_{t, \omega} \right\rVert_2, \label{eqn: avgConv5} \\
	\left\lVert \widetilde{\mathbf{b}}_{t, \omega} \right\rVert_2 &\leq \widehat{K}_{t, \omega} \Delta_{\mathcal{A}} \left(\frac{X_\omega}{(t+1)^{\tau_3}} + \frac{W_\omega}{(t+1)^{\frac{1}{2} - \epsilon_W}} \right).
\end{align}
Substituting for $\widetilde{\mathcal{G}}_{\mathcal{N}}^{t, \omega}$ and $\mathbf{b}_t$ into the dynamics of $\overline{\mathbf{e}}_{t}$ (equation~\eqref{eqn: averageErrorDynamics}) and taking the $\ell_2$-norm of both sides, we have
\begin{equation}\label{eqn: avgConv8}
	\begin{split}
		\left\lVert \overline{\mathbf{e}}_{t+1, \omega} \right\rVert_2 & \leq \left\lVert  \left( I_{M} - \alpha_t\mathcal{D} \widetilde{\mathcal{G}}^{t, \omega}_{\mathcal{N}} \right) \overline{\mathbf{e}}_{t, \omega} \!+\! \alpha_t \mathcal{D} \overline{\mathbf{b}}_{t, \omega} \right\rVert_2 \!\! + \frac{\alpha_t P X_\omega}{\underline{J}(t+1)^{\tau_3}} + \frac{\alpha_t P W_\omega}{\underline{J}(t+1)^{\frac{1}{2} - \epsilon_W}}.\!\!
	\end{split}
\end{equation}

From the resilience condition~\eqref{eqn: resilienceCondition} and the definition of $ \widetilde{\mathcal{G}}_{\mathcal{N}}^{t, \omega}$ we have $\lambda_{\min} \left( \widetilde{\mathcal{G}}_{\mathcal{N}}^{t, \omega}\right) \!>\! \widehat{K}_{t, \omega} \lambda_{\min} \left( \mathcal{G}_{\mathcal{N}} \right) > \widehat{K}_{t, \omega} \Delta_{\mathcal{A}}$, which, substituting into~\eqref{eqn: avgConv5} means that
\begin{align}\label{eqn: avgConv9}
	\left\lVert \overline{\mathbf{b}}_{t, \omega} \right\rVert_2 <\lambda_{\min} \left( \widetilde{\mathcal{G}}_{\mathcal{N}}^{t, \omega}\right)\left\lVert \overline{\mathbf{e}}_{t, \omega} \right\rVert_2.
\end{align}
As a result of~\eqref{eqn: avgConv9} and Lemma~\ref{lem: main}, there exists a positive definite matrix $\mathcal{G}^*_{t, \omega}$ such that $\mathcal{G}^*_{t, \omega} \overline{\mathbf{e}}_{t, \omega} = \widetilde{\mathcal{G}}_{\mathcal{N}}^{t, \omega} \overline{\mathbf{e}}_{t, \omega} + \overline{\mathbf{b}}_{t, \omega}$ with a minimum eigenvalue that satisfies
\begin{equation}\label{eqn: avgConv10}
	\lambda_{\min} \left(\mathcal{G}^*_{t, \omega} \right) \geq \lambda_{\min} \left(\widetilde{\mathcal{G}}_{\mathcal{N}}^{t, \omega} \right) - \widehat{K}_{t, \omega} \Delta_{\mathcal{A}} > \widehat{K}_{t, \omega}\kappa, 
\end{equation}
where, by definition~\eqref{eqn: kappaDef}, $\kappa = \lambda_{\min} \left( \mathcal{G}_{\mathcal{N}} \right) - \Delta_{\mathcal{A}}$. Substituting for $\mathcal{G}^*_{t, \omega} \overline{\mathbf{e}}_{t, \omega} = \widetilde{\mathcal{G}}_{\mathcal{N}}^{t, \omega} \overline{\mathbf{e}}_{t, \omega} + \overline{\mathbf{b}}_{t, \omega}$ in~\eqref{eqn: avgConv8} and using the fact that the matrix $\mathcal{D} \mathcal{G}_{t, \omega}^*$ is similar to the matrix $\mathcal{D}^{\frac{1}{2}} \mathcal{G}_{t, \omega}^* \mathcal{D}^{\frac{1}{2}}$, we have
\begin{equation}\label{eqn: avgConv11}
	\left\lVert \overline{\mathbf{e}}_{t+1, \omega} \right\rVert_2  \leq \left(1 - \frac{\alpha_t \widehat{K}_{t, \omega} \kappa}{\overline{J}} \right)\left\lVert \overline{\mathbf{e}}_{t, \omega} \right\rVert_2 + \frac{\alpha_t P R_\omega}{\underline{J}(t+1)^{\tau_4}},
\end{equation}
where $\tau_4 = \min\left(\tau_3, \frac{1}{2} - \epsilon_W \right)$, $R_\omega = X_\omega + W_\omega$, and $\overline{J}$ and $\underline{J}$ are the largest and smallest sizes, respectively, of the sets $\mathcal{J}_m$. By definition of $\widehat{K}_{t, \omega}$, we have \begin{equation}\label{eqn: avgConv12} \widehat{K}_{t, \omega} > {\Gamma}\left({ (t+1)^{\tau_\gamma} \left(\left\lVert \overline{\mathbf{e}}_{t, \omega} \right\rVert_2 + R_\omega \right) }\right)^{-1}. \end{equation}
Substituting for~\eqref{eqn: avgConv12} in its right hand side, we see that~\eqref{eqn: avgConv11} falls under the purview of Lemma~\ref{lem: timeVaryingSystem1} in the appendix, which means that
\begin{equation}\label{eqn: avgConv13}
	\lim_{t \rightarrow \infty} \left( t+1 \right)^{\tau_0} \left\lVert \overline{\mathbf{e}}_{t, \omega} \right\rVert_2 = 0,
\end{equation}
for every $0 \leq \tau_0 < \tau_4 - \tau_{\gamma} = \min \left( \tau_3 - \tau_\gamma, \frac{1}{2} - \epsilon_W - \tau_\gamma \right)$. Recall that, by definition, $\tau_3 = \tau_{\gamma} + \tau_1 - \tau_2 - \epsilon_X$ and $\tau_{\gamma} <\min \left(\frac{1}{2}, \tau_1 - \tau_2, 1-\tau_1\right) \leq \min \left(\frac{1}{2}, \tau_1 - \tau_2\right)$. Taking $\epsilon_X > 0 $ and $\epsilon_W > 0$ to be arbitrarily small, ~\eqref{eqn: avgConv13} holds for every $0 \leq \tau_0 < \min\left(\tau_\gamma, \frac{1}{2} - \tau_{\gamma}\right)$. The set of sample paths $\Omega^'$ for which the above analysis holds has probability measure $1$, yielding the desired result~\eqref{eqn: averageConv}.
\end{proof}

\subsection{Proof of Theorem~\ref{thm: main}}
\begin{proof}
From the triangle inequality, we have
\begin{align}\label{eqn: mainTriangle}
	\left\lVert x_n(t) - \theta^*_{\mathcal{I}_n} \right\rVert_2 & = \left\lVert Q_n \Big( \widetilde{x}_n(t) - \theta^* \Big) \right\rVert_2 \leq \left\lVert \mathcal{Q} \widehat{\mathbf{x}}_t \right\rVert_2 + \left\lVert \overline{\mathbf{e}}_t \right\rVert_2.
\end{align}
Lemmas~\ref{lem: resilientConsensus} and~\ref{lem: avg2} state that, if $\lambda_{\min} \left( \mathcal{G}_{\mathcal{N}} \right) > \Delta_{\mathcal{A}}$ (resilience condition~\eqref{eqn: resilienceCondition}), then
	$\mathbb{P}\Big( \lim_{t \rightarrow \infty} \left(t+1\right)^{\tau_3} \left\lVert \mathcal{Q} \widehat{\mathbf{x}}_t \right\rVert_2 = 0 \Big) = 1,$ and 
$\mathbb{P} \Big( \lim_{t \rightarrow \infty} (t+1)^{\tau_0} \left\lVert \overline{\mathbf{e}}_t \right \rVert_2 = 0 \Big) = 1,$ 
for every $0 \leq \tau_3 < \tau_\gamma + \tau_1 - \tau_2$ and every $0 \leq \tau_0 < \min\left(\tau_\gamma, \frac{1}{2} - \tau_{\gamma} \right)$. Combining these relations with~\eqref{eqn: mainTriangle} yields the desired result: for every $0 \leq \tau_0 < \min\left(\tau_\gamma, \frac{1}{2} - \tau_\gamma\right),$
\begin{equation}
	\mathbb{P} \left( \lim_{t \rightarrow \infty} \left(t+1 \right)^{\tau_0} \left\lVert x_n(t) - \theta^*_{\mathcal{I}_n} \right\rVert_2 = 0 \right) = 1.
\end{equation}
\end{proof}

\section{Numerical Examples}\label{sect: examples}
For numerical simulation, we consider a network of $625$ robots or agents (depicted in Figure~\ref{fig: network}) sensing an unknown two dimensional environment.\footnote{\color{black} Code for the numerical simulation is found at https://github.com/ychen824/SAFE.} {\color{black} The communication network is a two dimensional (25 node by 25 node) mesh lattice network.} We represent the environment as a $230$ unit by $230$ unit grid. {\color{black}Each location in the grid is represented by the tuple $(i, j)$, $i, j = 0, \dots, 229$. Let
\begin{equation}
	\mathcal{E} = \left\{(i, j) \vert i, j \in \left\{0, 1, \dots, 299 \right\} \right\}
\end{equation}
be the set of all location coordinates in the environment.} We assign each unit square {\color{black} $(i, j)$} in the environment a state value, which takes values on the interval $[0, 255]$. The state value represents, for example, the occupancy of a particular location: a value of $0$ may represent a location that is free of obstacles, while a value of $255$ may represent a location that is occupied by an impassable obstacle. The field parameter $\theta^*$ is the collection of state values for each of the grid locations. {\color{black} The state value of location $(i, j)$ is found at the $m_{(i, j)}^{\text{th}}$ component of $\theta^*$, where 
\begin{equation}\label{eqn: coordinateToComponent}
	m_{(i, j)} = 230\cdot i + j.
\end{equation}} The dimension of $\theta^*$ is $52,900$. 
\begin{figure}[h!]
	\centering
	\includegraphics[width = 0.5\columnwidth]{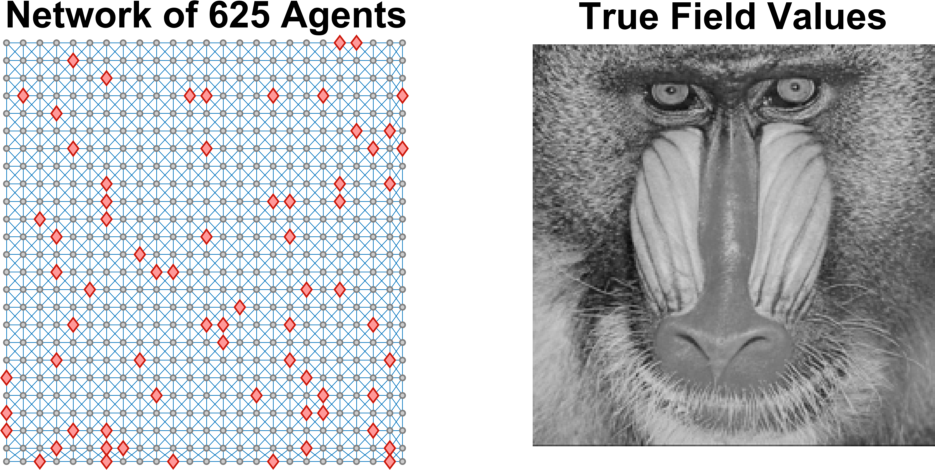}
	\caption{A team of $625$ agents, connected by a mesh communication network, (left) measures $\theta^*$, the state values of a $230$ by $230$ two dimensional grid environment, represented by the image of the baboon (right). Each pixel of the image represents the state value of a single location. The red diamonds represent agents with compromised measuremnt streams.}\label{fig: network}
\end{figure}
We represent the field $\theta^*$ using the $230$ pixel by $230$ pixel image of the baboon in Figure~\ref{fig: network}, where the intensity of each pixel represents the state value of a particular location.

{\color{black}The agents are placed uniformly throughout the environment. We place each agent $n$ ($n = 1, 2, \dots, N$) into the square $(i_n, j_n)$, where
\begin{equation}\label{eqn: agentPlacement}
	\begin{split}
		i_n &= \left\lfloor \frac{n-1}{230} \right\rfloor, \: j_n = \text{mod} \left(n-1, 230 \right).
	\end{split}
\end{equation} 
In~\eqref{eqn: agentPlacement}, the floor operator $\left \lfloor x \right\rfloor$ computes the largest integer smaller than or equal to $x$, and the modulo operator $\text{mod}(x, y)$ computes the remainder of dividing $x$ by $y$. }Each agent (robot) measures the state values of all locations in a $29$ unit by $29$ unit square, centered at its own location. {\color{black}That is, each agent $n$ measures the state value of all locations $(i, j)$ belonging to the set
\begin{equation}
	\widetilde{\mathcal{E}}_n =  \left\{(i, j) \in \mathcal{E} \Big\vert \left\lvert i - i_n \right\rvert \leq 29, \: \left\lvert j - j_n \right\rvert \leq 29 \right\}.
\end{equation} 
Equation~\eqref{eqn: coordinateToComponent} provides a one-to-one mapping between coordinates $(i, j) \in \mathcal{E}$ and field components $m \in \left\{1, \dots, M \right\}$. For a coordinate $(i, j)$, let $e^\intercal_{m_{(i, j)}}$, be the $m_{(i, j)}^{\text{th}}$ canoncial row vector of $\mathbb{R}^M$. Then, following~\eqref{eqn: noAttackMeasurement}, the measurement of agent $n$ (without attacks) is, 
\begin{equation}
	y_n(t) = H_n \theta^* + w_n(t),
\end{equation}
where the measurement matrix $H_n$ stacks all of the row vectors $e^\intercal_{m_{(i, j)}}$ for every $(i, j) \in\widetilde{ \mathcal{E}}_n.$ The physical coupling set of agent $n$ is
\begin{equation}\label{eqn: exampleCoupling}
	\widetilde{\mathcal{I}}_n = \left\{ m_{(i, j)} \Big\vert (i, j) \in \widetilde{\mathcal{E}}_n \right\},
\end{equation}
where, recall, from~\eqref{eqn: coordinateToComponent}, $m_{(i, j)} = 230 \cdot i + j$.}
For every agent, each of its scalar measurement streams is affected by i.i.d. (in time) Gaussian measurement noise with mean $0$ and variance $\sigma^2 = 50$. The noise is independent across measurement streams.


Each agent is interested in the state values of all locations in a $57$ unit by $57$ unit square, centered at its own location, {\color{black} i.e., each agent $n$ is interested in estimating the state values of all locations in the set 
\begin{equation}
	\mathcal{E}_n = \left\{(i, j) \in \mathcal{E} \Big\vert \left\lvert i - i_n \right\rvert \leq 57, \: \left\lvert j - j_n \right\rvert \leq 57 \right\}.
\end{equation}
The interest of agent $n$ is
\begin{equation}\label{eqn: exampleInterestSet}
	\mathcal{I}_n = \left\{ m_{(i, j)} \Big\vert (i, j) \in {\mathcal{E}}_n \right\}.
\end{equation}}Note that each agent is interested in components of the field $\theta^*$ that it does not directly measure{\color{black}, i.e., there exist components $m \in \mathcal{I}_n$ that do not belong to $\widetilde{\mathcal{I}}_n$. Thus, agents must rely on communication with neighbors to estimate their components of interest.} An adversary compromises all of the measurement streams of $70$ agents ($11.2\%$ of the agents), chosen uniformly at random without replacement, shown as the red diamonds in Figure~\ref{fig: network}. The adversary changes the value of each measurement stream under attack ($p \in \mathcal{A}$) to $y^{(p)}(t) = 255$. 

We compare the performance, averaged over $100$ trials, of $\textbf{SAFE}$ against the performance of $\textbf{CIRFE}$, the distributed field estimator from~\cite{SahuRandomFields}, which did not consider measurement attacks. We use the following weights: $a = 1, b = 0.0839, \tau_1 = 0.26, \tau_2 = 0.001, \Gamma = 40, \tau_\gamma = 0.25.$
\begin{figure}
	\centering
	\begin{subfigure}[h]{0.45\textwidth}
	\includegraphics[width = \columnwidth]{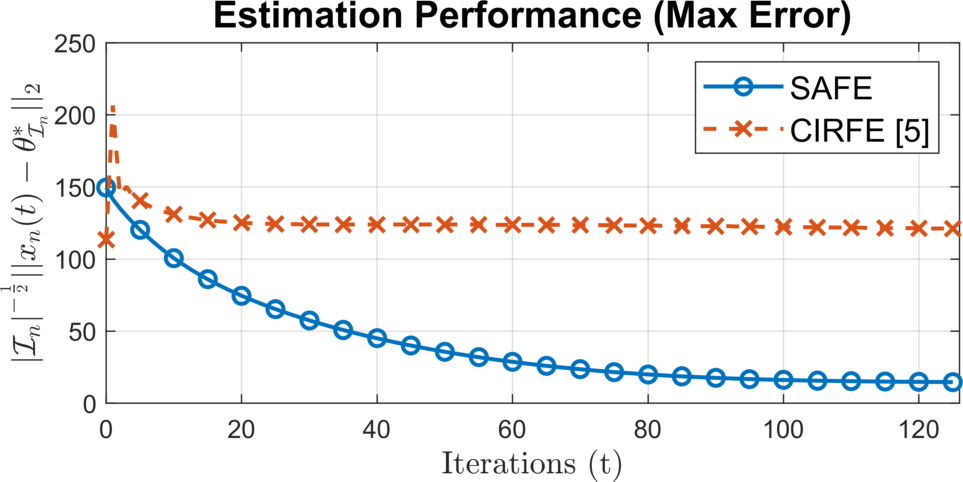}
	\caption{}\label{fig: performance1}
	\end{subfigure}
\hspace{.05\textwidth}
\begin{subfigure}[h ]{0.45 \textwidth}
	\centering
	\includegraphics[width = \columnwidth]{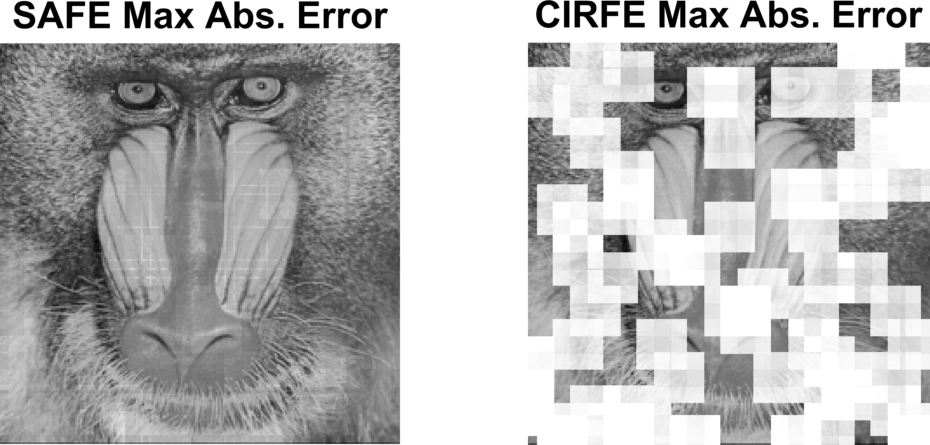}
	\caption{}\label{fig: performance2}
	\end{subfigure}
	\caption{(a) Evolution of the maximum root mean square error (RMSE) over all agents, normalized by square root of the size of the agent's interest set, for \textbf{SAFE} and \textbf{CIRFE}~\cite{SahuRandomFields}. (b) Estimation results for \textbf{SAFE} (left) and \textbf{CIRFE}~\cite{SahuRandomFields} (right) under measurement attacks.}
\end{figure}
Figure~\ref{fig: performance1} shows the evolution of the maximum estimation root mean squared error (RMSE) over all agents, normalized by the square root of the size of each agent's interest set, for \textbf{SAFE} and \textbf{CIRFE}, the estimator from~\cite{SahuRandomFields}. Under \textbf{SAFE}, the agents' local estimation error converges to zero, even when there is an adversary. In contrast, under the estimator from~\cite{SahuRandomFields}, the adversary causes a significant error in some of the agents' estimates. Figure~\ref{fig: performance2} shows the reconstructed field from \textbf{SAFE} and the estimator from~\cite{SahuRandomFields}, where, for each component of the field (i.e., for each pixel in the image), we take the worst estimate among all agents intersted in that component. Similar to Figure~\ref{fig: performance1}, Figure~\ref{fig: performance2} shows that \textbf{SAFE} allows the agents consistently estimate the components of the field in which they are interested while under adversarial attack. The same adversarial attacks induce significant errors in the estimation results of agents following \textbf{CIRFE}, an estimator that does not have proper countermeasures against measurement attacks.

\section{Conclusion}\label{sect: conclusion}
In this paper, we have studied resilient distributed field estimation under measurement attacks. A team of agents or devices, connected by a cyber communication network, makes measurements of a large, spatially distributed physical field, and each agent processes the measurements to recover specific components of the field. For example, in multi-robot navigation, a team of connected robots takes measurements of an unknown environment and each robot processes its measurements to estimate its local surroudings. An adversary arbitrarily manipulates a subset of the measurements. To deal with the measurement attacks, this paper presented \textbf{SAFE}, the Saturating Adaptive Field Estimator, a resilient consensus+innovations distributed field estimator. Under \textbf{SAFE}, each agent applies an adaptive gain to its innovations component to saturate its magnitude at a time-decaying threshold. As long as the subnetwork of agents interested in each component of the field is connected on average, and there is enough measurement redundancy in the noncompromised measurement streams, then, \textbf{SAFE} guarantees that all of the agents' local estimates converge almost surely true value of the components of the field in which they are interested. Finally, we illustrated the performance of \textbf{SAFE} through numerical examples.

\section{Appendix}
We require intermediate results from~\cite{ChenDistributed1, ChenSAGE, Schur} to prove Theorem~\ref{thm: main}. {\color{black}For completeness, we provide proof sketches for Lemmas \ref{lem: consensus} and~\ref{lem: timeVaryingSystem1}, originally from~\cite{ChenSAGE}.}

\subsection{Consensus Analysis}
We use the following result, a consequence of Lemma~1 in~\cite{ChenSAGE} to study the convergence of the agents' estimates to the generalized network average estimate. 
\begin{lemma}[Lemma 1 in~\cite{ChenSAGE}]\label{lem: consensus}
	Define the consensus subspace
\begin{equation}\label{eqn: consensusSubspace}
	\mathcal{C} = \left\{ w \in \mathbb{R}^{NM} \vert w = \mathbf{1}_N \otimes v , v \in \mathbb{R}^M \right\},
\end{equation}
and let $\mathcal{C}^\perp$ be the orthogonal complement of $\mathcal{C}$. Let $\widehat{\mathbf{w}}_t \in \mathcal{C}^\perp$ be $\mathcal{F}_t$-adapted and evolve according to
\begin{equation}\label{eqn: wHatEvolution}
	\widehat{\mathbf{w}}_{t+1}= \left(I_{NM} - P_{NM} - r_2(t) L(t) \right) \widehat{\mathbf{w}}_t + r_1(t) \mathbf{v}_t,
\end{equation}
where:
\begin{enumerate}
	\item $P_{NM} = \frac{1}{N} \left(\mathbf{1}_N \mathbf{1}_N^\intercal \right) \otimes I_M$,
	\item the sequences $\left\{ r_1(t) \right\}$ and $\left\{ r_2(t) \right\}$ follow
		$r_1(t) = \frac{c_1}{(t+1)^{\delta_1}}, \: r_2(t) = \frac{c_2}{(t+1)^{\delta_2}},$
	with $c_1, c_2 > 0$, $0 < \delta_2 < \delta_1$, and $\delta_2 < 1$. 
	\item the term $\mathbf{v}_t$ satisfies
		$\left\lVert \mathbf{v}_t \right\rVert_2 \leq \frac{c_3}{(t+1)^{\delta_3}},$
	with $0 < \delta_3 < \min\left(\frac{1}{2}, \delta_1 - \delta_2 \right)$ and $c_3 > 0$, and
	\item the sequence $\left\{ L(t) \right\}$ is an $\mathcal{F}_{t+1}$-adapted, i.i.d. sequence of (undirected) graph Laplacian matrices that satisfies $\lambda_2 \left(\mathbb{E} \left[ L(t) \right] \right) > 0$. 
\end{enumerate}
Then, for every $0 \leq \delta < \delta_3 + \delta_1 - \delta_2$, $\widehat{\mathbf{w}}_t$ satisfies
\begin{equation}
	\mathbb{P} \left(\lim_{t \rightarrow \infty} \left(t+1\right)^\delta \left\lVert \widehat{\mathbf{w}}_t \right\rVert_2 = 0 \right) = 0.
\end{equation}

\end{lemma}

{\color{black} The proof sketch of Lemma~\ref{lem: consensus} requires two results from~\cite{Kar4}. The following result (Lemma 4.4 in~\cite{Kar4}) characterizes the effect of random, time-varying Laplacians on the evolution of $\widehat{\mathbf{w}}_t$ in~\eqref{eqn: wHatEvolution}. 
\begin{lemma}[Lemma 4.4 in~\cite{Kar4}]\label{lem: randomLaplacians}
	Consider $\widehat{\mathbf{w}}_t$ defined in~\eqref{eqn: wHatEvolution}, subject to the conditions 1. -- 4. stated in Lemma~\ref{lem: consensus}. There exists a measurement $\mathcal{F}_{t+1}$-adapted process $\left\{ r(t) \right\}$ and a constant $c_r > $ such that, for $t$ large enough,
\begin{align}
	0 \leq r(t) &\leq 1, \text{ a.s.}, \label{eqn: randLaplacianConditions1}\\
	\frac{c_r}{(t+1)^{\delta_2}} &\leq \mathbb{E} \left[ r(t) \vert \mathcal{F}_{t} \right], \label{eqn: randLaplacianConditions2}\\
	\left\lVert \left(I_{NM} - r_2(t) L(t) \right) \widehat{\mathbf{w}}_t \right\rVert_2 &\leq \left(1- r(t)\right) \left\lVert \widehat{\mathbf{w}}_t \right\rVert_2, \text{ a.s.},
\end{align}
where the sequence $\left\{ r_2(t) \right\}$ and the constant $\delta_2$ satisfy condition 2. from Lemma~\ref{lem: consensus}. 
\end{lemma}

The following result from~\cite{Kar4} studies the evolution of scalar dynamical systems of the form
\begin{equation}\label{eqn: randomDynamicalSystem}
	w_{t+1} = \left(1-r(t) \right) w_t + r_1(t),
\end{equation}
where $\left\{ r_1(t) \right\}$ is a deterministic sequence that follows, $r_1(t) = \frac{c_1}{(t+1)^{\delta_1}}$, $\left\{r(t)\right\}$ is $\mathcal{F}_{t+1}$-adapted process satisfying $0 \leq r(t) \leq 1, \text{ a.s.}$ and $\frac{c_r}{(t+1)^{\delta_2}} \leq \mathbb{E} \left[ r(t) \vert \mathcal{F}_{t} \right]$, and the constants $c_1, c_r, \delta_1, \delta_2$ satisfy, $c_1, c_r > 0$, $0 < \delta_2 < \delta_1$, $\delta_2 < 1$.

\begin{lemma}[Lemma 4.2 in~\cite{Kar4}]\label{lem: randomDynamicalSystem}
	The system in~\eqref{eqn: randomDynamicalSystem} satisfies
	\begin{equation}
		\mathbb{P} \left( \lim_{t \rightarrow \infty} (t+1)^{\delta_0} w_t = 0 \right) = 1,
	\end{equation}
	for every $0 \leq \delta_0 < \delta_1 - \delta_2$. 
\end{lemma}

For completeness, we now provide the proof sketch for Lemma~\ref{lem: consensus}. The full proof is found in~\cite{ChenSAGE}. 

\begin{proof}[Proof of Lemma~\ref{lem: consensus}]
Consider the dynamics of $\widehat{\mathbf{w}}_t$ (equation~\eqref{eqn: wHatEvolution}). Since, by definition, $\widehat{\mathbf{w}}_t \in \mathcal{C}^\perp$, we have
\begin{equation}\label{eqn: consensusProof1}
	P_{NM} \widehat{\mathbf{w}}_t = 0.
\end{equation}
Then, substituting~\eqref{eqn: consensusProof1} into~\eqref{eqn: wHatEvolution}, taking the $\ell_2$-norm of both sides, and applying the triangle inequality, we have
\begin{equation}\label{eqn: consensusProof2}
	\left\lVert \widehat{\mathbf{w}}_{t+1} \right\rVert_2 \leq \left\lVert \left(I_{NM} - r_2(t) L(t)\right)  \widehat{\mathbf{w}}_t\right\rVert_2 + r_1(t) \left\lVert \mathbf{v}_t \right\rVert_2.
\end{equation}
From conditions 1. and 3. stated in Lemma~\ref{lem: consensus}, we have 
\begin{equation}\label{eqn: consensusProof3}
	r_1(t) \left\lVert \mathbf{v}_t \right\rVert_2 \leq \frac{c_1 c_3}{(t+1)^{\delta_1 + \delta_3}}.
\end{equation}
From Lemma~\ref{lem: randomLaplacians}, we have
\begin{equation}\label{eqn: consensusProof4}
	\left\lVert \left(I_{NM} - r_2(t) L(t) \right) \widehat{\mathbf{w}}_t \right\rVert_2 \leq \left(1- r(t)\right) \left\lVert \widehat{\mathbf{w}}_t \right\rVert_2,
\end{equation}
where $r(t)$ satisfies~\eqref{eqn: randLaplacianConditions1} and~\eqref{eqn: randLaplacianConditions2}. Substituting~\eqref{eqn: consensusProof3} and~\eqref{eqn: consensusProof4} into~\eqref{eqn: consensusProof2} yields
\begin{equation}\label{eqn: consensusProof5}
	\left\lVert \widehat{\mathbf{w}}_{t+1} \right\rVert_2 \leq  \left(1- r(t)\right) \left\lVert \widehat{\mathbf{w}}_t \right\rVert_2 + \frac{c_1 c_3}{(t+1)^{\delta_1 + \delta_3}}.  
\end{equation}
The relation in~\eqref{eqn: consensusProof5} falls under the purview of Lemma~\ref{lem: randomDynamicalSystem}, which yields the desired result.
\end{proof}}

\subsection{Postive Semi-Definite Block Matrices}
We use the following result from~\cite{Schur} to establish the positive semi-definiteness of symmetric matrices.
\begin{lemma}[Proposition 16.1 in~\cite{Schur}]\label{lem: schur}
	Let $\mathcal{M}$ be a symmetric matrix of the form
	$\mathcal{M} = \left[\begin{array}{cc} A & B^\intercal \\ B & C \end{array} \right].$
If $C \succ 0$, then, $\mathcal{M} \succeq 0$ if and only if $A - B^\intercal C^{-1}B \succeq 0$. 
\end{lemma}

\subsection{Convergence of Averaged Noise}
The following result from~\cite{ChenDistributed1} characterizes the behavior of the averaged measurement noise in the agents' time-averaged measurements.
\begin{lemma}[Lemma 5 in~\cite{ChenDistributed1}]\label{lem: measureNoise}
	Let $v_1, v_2, v_3, \dots$ be i.i.d. random variables with mean $\mathbb{E} \left[ v_t \right] = 0$ and finite covariance $\mathbb{E} \left[ v_t v_t^{\intercal} \right] = \Sigma.$ 
Then, we have
\begin{equation}\label{eqn: noiseConvergence}
	\mathbb{P} \left( \lim_{t \rightarrow \infty} (t+1)^{\delta_0} \left\lVert \overline{v}_t \right\rVert_2 = 0 \right) = 1,
\end{equation}
where $\overline{v}_t = \frac{1}{t+1} \sum_{j = 0}^t v_j$ for all $0 \leq \delta_0 < \frac{1}{2}.$
\end{lemma}

\subsection{Generalized Average Analysis}
We use the following Lemma, a consequence of Lemma 3 in~\cite{ChenSAGE}, to study the behavior of the generalized network average estimate. 
\begin{lemma}[Lemma 3 in~\cite{ChenSAGE}]\label{lem: timeVaryingSystem1}
Consider the scalar, time-varying dynamical system
\begin{equation}\label{eqn: timeVaryingSystem1}
	w_{t+1} = \left(1 - \frac{r_1(t) c_3}{\left(\left\lvert w_{t} \right\rvert + \widehat{c}_3 \right)(t+1)^{\delta_3}} \right) w_t + \frac{r_1(t) c_4}{(t+1)^{\delta_4}},
\end{equation}
where $r_1(t) = \frac{c_1}{(t+1)^{\delta_1}}$, $c_3, \widehat{c}_3, c_4 > 0$, $ 0 < \delta_3 < \delta_4 < \delta_1$, and $\delta_3 + \delta_1 < 1$. Then, for every $0 \leq \delta_0 < \delta_4 - \delta_3$, the system state $w_t$ satisfies
\begin{equation}\label{eqn: timeVaryingConvergence}
	\lim_{t \rightarrow \infty} \left( t+1 \right)^{\delta_0} w_t = 0.
\end{equation}
\end{lemma}

{\color{black}For completeness, we outline the proof of  Lemma~\ref{lem: timeVaryingSystem1}. 

\begin{proof}[Proof of Lemma~\ref{lem: timeVaryingSystem1}]
First, we analyze scalar, time-varying dynamical systems of the form
\begin{equation}\label{eqn: timeVaryProof1}
\begin{split}
	\widehat{m}_{t+1} &= \left(1 - \frac{r_2(t)}{m_t + c_5} \right) m_t + r_1(t), \\
	m_{t+1} &= \max \left(\left\lvert m_t \right\rvert, \left\lvert \widehat{m}_{t+1} \right\rvert \right),
\end{split}
\end{equation}
with initial condition $m_0 \geq 0$, where $r_1(t) = \frac{c_1}{(t+1)^{\delta_1}}$, $r_2(t) = \frac{c_2}{(t+1)^{\delta_2}}$, and $c_1, c_2, c_5 > 0$. Specifically, we show that, under~\eqref{eqn: timeVaryProof1}, $m_t$ satisfies
\begin{equation}\label{eqn: timeVaryProof2}
	\sup_{t \geq 0} m_t < \infty.
\end{equation}
We show that there exists finite $T > 0$ such that $m_{T+1} = m_{T}$. 

Note that, by consturction, $m_t$ is non-negative and non-decreasing, and $m_1 > 0$. By definition, there exists $T$ large enough such that $\frac{r_2(T)}{m_T + c_5} \leq 1$, which ensures that $\left\lvert \widehat{m}_{T+1} \right\rvert = \widehat{m}_{T+1}$. Then, for $T$ large enough, a sufficient condition for $m_{T+1} = m_{T}$ is $m_T - \widehat{m}_{T+1} \geq 0$. By performing algebraic manipulations on~\eqref{eqn: timeVaryProof1}, we may show that
\begin{equation}\label{eqn: timeVaryProof3}
	m_T - \widehat{m}_{T+1} = r_2(T) \frac{m_T}{m_T + c_5} - r_1(T).
\end{equation}
The term $\frac{m_T}{m_T + c_5}$ increases in $m_T$. Recall that $m_T$ nondecreasing and $m_1 > 0$, so we have
\begin{equation}\label{eqn: timeVaryProof5}
	\frac{m_T}{m_T + c_5} \geq \frac{m_1}{m_1 + c_5} > 0.
\end{equation}

Now, let $c_6 = \frac{m_1}{m_1 + c_5}$. Substituting into~\eqref{eqn: timeVaryProof5}, and performing algebraic manipulations, we may show that sufficient condition for $m_T - \widehat{m}_{T+1} \geq 0$ is
\begin{equation}\label{eqn: timeVaryProof6}
	T \geq \left(\frac{c_1}{c_2 c_6}\right)^{\frac{1}{\delta_1 - \delta_2}} - 1. 
\end{equation}
The right hand side of~\eqref{eqn: timeVaryProof6} is finite, so there exists finite $T$ such that $m_{T+1} = m_T$. Since~\eqref{eqn: timeVaryProof6} is a \textit{sufficient} condition, if the right hand side negative, then, taking $T = 1$ satisfies $m_{T+1} = m_{T}$. Further note that all $t \geq T$ also satisfy~\eqref{eqn: timeVaryProof6}, which means that $m_T = m_{T+1} = m_{T+2} = \dots.$ Then, we have $\sup_{t \geq 0} m_t = m_T < \infty$, yielding, ~\eqref{eqn: timeVaryProof2}. 

Second, we now use~\eqref{eqn: timeVaryProof1} and~\eqref{eqn: timeVaryProof2} to analyze $\sup_{0 \leq j \leq t} \left\lvert w_j \right\rvert$. Define the scalar, time-varying dynamical system
\begin{equation}\label{eqn: timeVaryProof7}
\begin{split}
	\widehat{v}_{t+1} &=  \left(1 - \frac{r_1(t) c_3}{\left(\left\lvert v_{t} \right\rvert + \widehat{c}_3 \right)(t+1)^{\delta_3}} \right) v_t + \frac{r_1(t) c_4}{(t+1)^{\delta_4}}, \\
	v_{t+1}& = \max \left(\left\lvert v_t \right\rvert, \left\lvert \widehat{v}_{t+1} \right\rvert \right),
\end{split}
\end{equation}
where $r_1(t), r_2(t), c_3, \widehat{c}_3, c_4, \delta_3,$ and $\delta_4$ have the same values as they do in~\eqref{eqn: timeVaryingSystem1}, with initial condition $v_0 = \left\lvert w_0 \right\rvert$. By construction, we have
\begin{equation}\label{eqn: timeVaryProof8}
	v_t \geq \sup_{0 \leq j \leq t} \left\lvert w_j \right\rvert,
\end{equation} 
for all $t \geq 0$. The system in~\eqref{eqn: timeVaryProof7} is of the same form as the system in~\eqref{eqn: timeVaryProof1}, which, from~\eqref{eqn: timeVaryProof2}, means that
\begin{equation}\label{eqn: timeVaryProof9}
	\sup_{t \geq 0} \left\lvert w_t \right \rvert \leq \sup_{t \geq 0} v_t < \infty.
\end{equation}

Third, letting $\widehat{w} = \sup_{t \geq 0} \left\lvert w_t \right\rvert$, from~\eqref{eqn: timeVaryingSystem1}, for $t$ large enough, we have
\begin{equation}\label{eqn: timeVaryProof10}
	w_{t+1} \leq \left(1 - \frac{r_1(t) c_3}{\left(\widehat{w} + \widehat{c}_3 \right)(t+1)^{\delta_3}} \right) w_t + \frac{r_1(t) c_4}{(t+1)^{\delta_4}}.
\end{equation}
The relation in~\eqref{eqn: timeVaryProof10} falls under the purview of Lemma~\ref{lem: randomDynamicalSystem}, which yields the desired result~\eqref{eqn: timeVaryingConvergence}. 
\end{proof}}

\bibliography{IEEEabrv,References}

\end{document}